\documentclass[USenglish]{article}
\usepackage[utf8]{inputenc}

\usepackage{amssymb,amsmath,amsthm,amscd,epsf,latexsym,verbatim,graphicx,amsfonts,hyperref,epstopdf,xcolor,wasysym}
\input epsf.tex
\usepackage[utf8]{inputenc}
\usepackage[T1]{fontenc}
\usepackage{graphicx}
\usepackage{palatino, url, multicol}
\usepackage{subfig}
\usepackage{geometry}                
\usepackage{hyperref}
\usepackage{afterpage}
\theoremstyle{theorem}
\newtheorem{theorem}{Theorem}[section]

\newtheorem{lemma}[theorem]{Lemma}
\newtheorem{proposition}[theorem]{Proposition}
\newtheorem{corollary}[theorem]{Corollary}

\theoremstyle{definition}

\theoremstyle{definition}

\theoremstyle{definition}
\newtheorem{remark}[theorem]{Remark}
\theoremstyle{definition}
\newtheorem{example}[theorem]{Example}
\theoremstyle{definition}
\newtheorem{definition}[theorem]{Definition}
\theoremstyle{definition}

\theoremstyle{definition}
\newtheorem{question}[theorem]{Question}
\theoremstyle{question}

\title{Periodicity and ergodicity in the trihexagonal tiling}
\author{Diana Davis and W. Patrick Hooper}

\newif\ifdraft
\draftfalse

\newcommand{\compat}[1]{{\ifdraft{\textcolor{violet}{\textrm{{\bf Pat says} {``}#1{''}}}}\else\ignorespaces\fi}}

\newcommand{\comref}[1]{{\ifdraft{\textcolor{red}{\textrm{{\bf Referee says} {``}#1{''}}}}\else\ignorespaces\fi}}

\renewcommand{\Im}{\operatorname{Im}}

\def\0{{\mathbf 0}}
\def\c{{\mathbf c}}
\def\e{{\mathbf v}}
\def\u{{\mathbf u}}
\def\v{{\mathbf v}}
\def\w{{\mathbf w}}
\def\x{{\mathbf x}}

\def\C{{\mathbb C}}
\def\bs{{\setminus}}
\def\G{{\mathbb G}}
\def\H{{\mathbb H}}

\def\R{{\mathbb R}}
\def\Z{{\mathbb Z}}
\def\RP{{{\mathbb R}{\mathbb P}}}
\def\PGL{{\mathrm{PGL}}}
\def\PSL{{\mathrm{PSL}}}
\def\GL{{\mathrm{GL}}}
\def\O{{\mathrm{O}}}
\def\SO{{\mathrm{SO}}}
\def\SL{{\mathrm{SL}}}
\def\hol{{\mathbf{hol}}}
\def\Area{{\mathit{Area}}}
\def\Aff{{\mathit{Aff}}}
\def\Trans{{\mathit{Trans}}}
\newcommand{\twotwo}[4]{\left[\begin{array}{rr} #1 & #2 \\ #3 & #4 \end{array}\right]}
\def\vis{{\mathit{vis}}}
\def\dev{{\mathit{dev}}}

\begin{document}
\maketitle

\begin{abstract}
We consider the dynamics of light rays in the trihexagonal tiling where triangles and hexagons are transparent and have equal but opposite indices of refraction. We find that almost every ray of light is dense in a region of a particular form: the regions have infinite area and consist of the plane with a periodic family of triangles removed. We also completely describe initial conditions for periodic and drift-periodic light rays.
\end{abstract}

\section{Introduction}

Consider a partition of the plane into regions that are each made up of one of two different transparent materials
so that the refraction coefficient for light traveling between the two materials is $-1$. This means that the two materials have indices of refraction with equal magnitude but opposite sign,
and ensures that a light ray exiting a region made of one material making an angle of $\theta$ with the normal to the boundary, enters a region made of the other material making an angle of $-\theta$ with the normal to the boundary. See the left side of Figure \ref{fig:trajectory}.
Materials with negative index of refraction were discovered about 15 years ago and have been heavily studied since, see \cite{shelby,smith}. The connection between these materials and planar tilings was made in \cite{bloch}. If it were possible to create metamaterials in sufficiently large quantities, we could actually construct our tiling out of these materials, shoot a laser through it, and observe the behaviors that we discuss here.

\begin{figure}
\begin{center}
\includegraphics[height=2in]{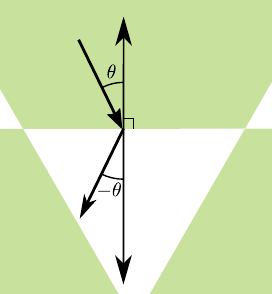}
\hspace{0.5in}
\includegraphics[height=2in]{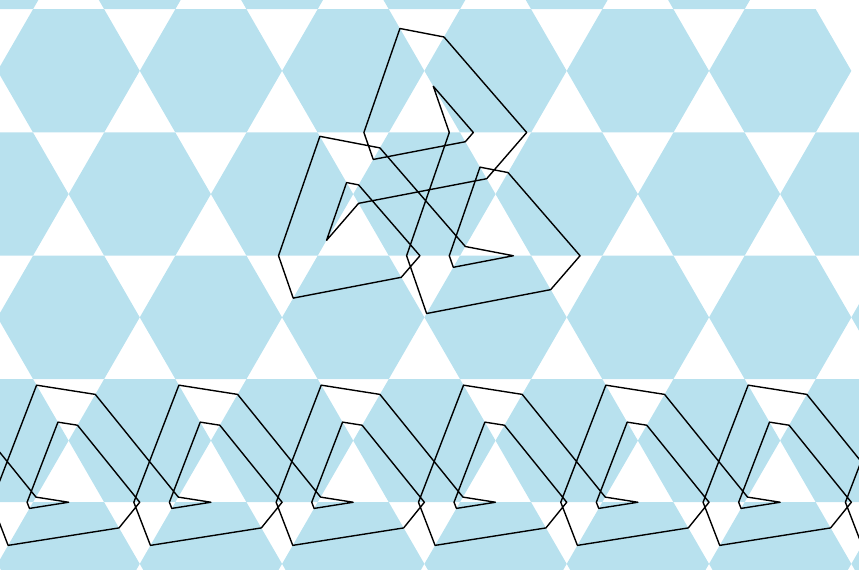}
\caption{Left: a light ray passing between media with opposite refraction indices.
Right: periodic and drift-periodic trajectories in the trihexagonal tiling.}
\label{fig:trajectory}
\end{center}
\end{figure}

The \emph{trihexagonal tiling} is the edge-to-edge tiling where an equilateral triangle and a regular hexagon meet at each edge. We consider the behavior of a light beam in such a tiling where the triangles are made with one material
and hexagons are made of a second material with opposite index of refraction.
See the right side of Figure \ref{fig:trajectory} for two light beams.
Trajectories exhibit a range of behaviors.
They may be periodic or {\em drift-periodic} (invariant under a non-trivial translational symmetry of the tiling). Such behaviors have been seen before in a number of tilings \cite{icerm}. However we find that a randomly chosen trajectory exhibits a new kind of behavior. Say that a trajectory {\em illuminates} a point in the plane if the point lies in the closure of the trajectory.

\begin{theorem}
\label{thm:1}
For almost every initial point and direction, the trajectory with this initial position and direction will illuminate all of the plane except a periodic family of triangular open sets in centers of either the upward-pointing triangles or downward-pointing triangles.
\end{theorem}

See Figure \ref{fig:dark-triangles} for an example of a portion of a trajectory which appears to fill part of the plane, but misses a periodic family of open triangles in the center of upward-pointing triangles.
Theorem \ref{thm:1} is a direct consequence of Theorem \ref{thm:main ergodic} below; see the end of \S \ref{sect:applications} for the proof.

\begin{figure}
\begin{center}
\includegraphics[height=3in]{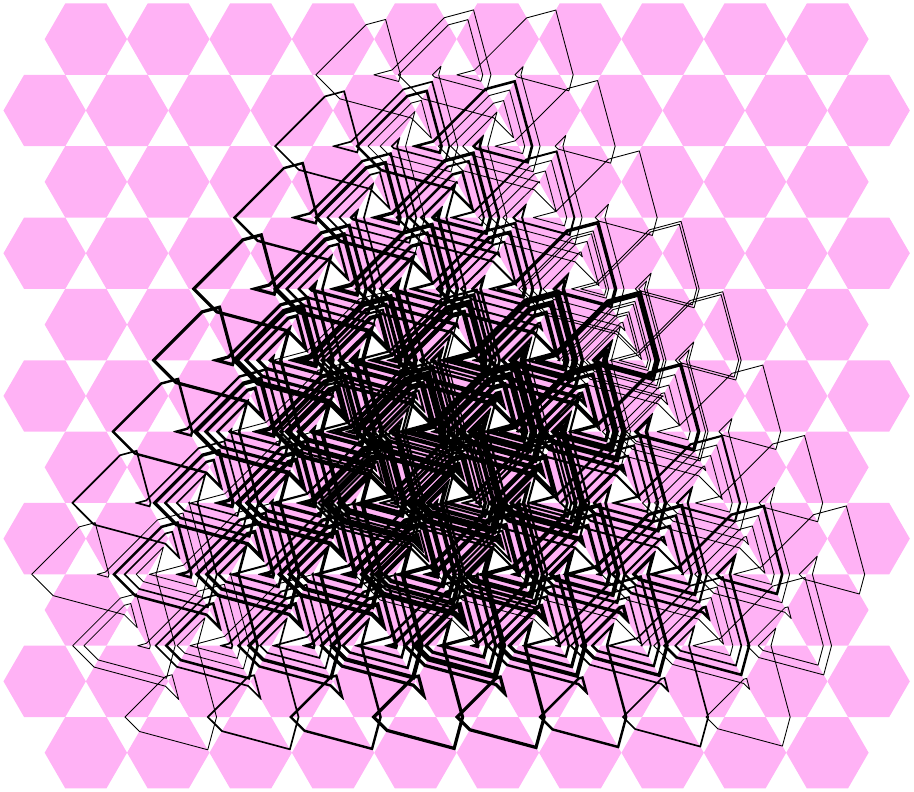}
\hspace{0.5in}
\includegraphics[height=3in]{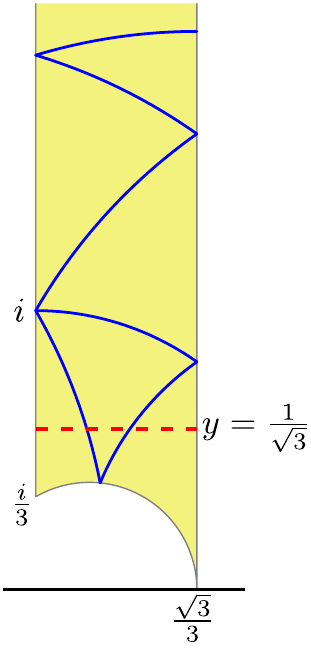}
\caption{Left: A portion of a trajectory initiating at angle $\theta=\frac{\pi}{4}$ from a vertex of a hexagon.
Note the untouched triangle centers. Right: The hyperbolic billiard table $\Delta$ shown in yellow with a blue periodic billiard path starting at $i$ at angle of $\frac{5\pi}{6}$ from the vertical.}
\label{fig:dark-triangles}
\end{center}
\end{figure}

We chose to study the trihexagonal tiling because of the apparent complexity of trajectories
as illustrated in Figure \ref{fig:dark-triangles}. Simpler tilings such as the three edge-to-edge tilings of the plane by regular $n$-gons (for $n \in \{3,4,6\}$) have only periodic and drift-periodic trajectories \cite{2012} \cite{bloch} \cite[Theorem 2.1]{icerm}.
We provide some further context for the study of this system in \S \ref{sect:context}.

By rescaling time, we may assume that light moves at unit speed as measured with respect to the Euclidean metric on the plane.
Then the motion of light defines what we call the {\em refractive} or {\em billiard flow} on the tiling, a unit speed flow $T^t:X \to X$
on the unit tangent bundle $X$ of the plane with singularities at the vertices of the tiling.
Trajectories are not defined through singularities. A trajectory is {\em non-singular} if it is defined for all time.

The behavior of a trajectory is determined to a large extent by the initial direction of travel. To formally state results of this form, we need a few basic observations about the behavior of trajectories. First, a trajectory
initially traveling in direction $\theta$ in in a hexagon can later only be traveling though a hexagon in a direction from the set $\{\theta, \theta+\frac{2\pi}{3}, \theta+\frac{4\pi}{3}\}$ and only be traveling through a triangle in directions from $\{-\theta, \frac{2\pi}{3}-\theta, \frac{4\pi}{3}-\theta\}$.
Here, by {\em direction of travel} we mean the signed angle a tangent vector to the trajectory makes with the horizontal (rightward) vector field.
Second, trajectories through the center of a hexagon hit singularities in both forward and backward time. Thus, a non-singular trajectory in a hexagon misses the center and travels in a counter-clockwise (positive) or clockwise (negative) direction around this center.
This notion of clockwise/counter-clockwise turns out to be flow invariant. This means that we can extend this notion of direction of travel around the center of a hexagon to trajectories within triangles by flowing until we enter a hexagon and then evaluating direction of travel there.
Given the choice of an angle $\theta$ and a sign $s=+$ or $s=-$, we denote by $T_{\theta,s}:X_{\theta,s} \to X_{\theta,s}$ the restriction of $T$ to $X_{\theta,s}$, the set of unit vectors in the plane traveling in directions as listed above
and traveling with sign $s$ around the centers of hexagons.
The domains $X_{\theta,s}$ have natural invariant measure which we call their Lebesgue measures because they arise
from Lebesgue measure on $\R^2$. The observations made above are formally described in \S \ref{sect:applications}.

Above we described the notion of direction as an angle in $\R/2 \pi \Z$, but it is also natural to think of this set of directions as identified with the unit circle
$S^1 \subset \R^2$. We abuse notation by identifying $S^1$ with $\R/2\pi \Z$.

The different kinds of behaviors observed in the systems $T_{\theta,s}:X_{\theta,s} \to X_{\theta,s}$
as we vary $\theta$ are related to dynamics on a triangular billiard table in the hyperbolic plane $\H^2$. This billiard table $\Delta$ is $\H^2$ modulo a $(3,\infty, \infty)$-triangle reflection group. For us $\Delta$ represents the specific table depicted on the right side of Figure \ref{fig:dark-triangles} as a triangle in the upper half-plane model of $\H^2$.

Let $g_t:T_1 \Delta \to T_1 \Delta$ denote the unit speed billiard flow on $\Delta$. The point $i \in \C$ sits in the boundary of $\Delta$ along a vertical wall. For $\theta \in S^1$, let $\vec{u}_\theta \in T^1 \Delta$ be the unit tangent vector which is tangent at $i$ to the geodesic ray initiating at $i$ and terminating at $|\!\cot \theta|$ in the boundary of the upper half-plane.
We define ${\mathcal E_0}$ to be the set of $\theta \in [\frac{\pi}{3}, \frac{2\pi}{3}]$ for which the forward billiard orbit
$\{g_t(\vec{u}_\theta)~:~t\geq 0\}$ has an accumulation point in the portion of $\Delta$ with imaginary part strictly greater than $\frac{1}{\sqrt{3}}$. This is the dashed line in Figure \ref{fig:dark-triangles}. In other words,
\begin{equation}
\label{eq:E0}
{\mathcal E}_0=\left\{\theta \in \left[\frac{\pi}{3}, \frac{2\pi}{3}\right]~:~
\limsup_{t \to +\infty} \Im\big(g_t(\vec{u}_\theta)\big) > \frac{1}{\sqrt{3}}
\quad \text{and} \quad
\liminf_{t \to +\infty} \Im\big(g_t(\vec{u}_\theta)\big) \neq +\infty
\right\},
\end{equation}
where $\Im\big(g_t(\vec{u}_\theta)\big)$ here denotes taking the imaginary part of the basepoint of a unit tangent vector in the upper half-plane. \compat{Changed notation from $\mathfrak{I}$ to $\Im$.}
By classic results about geodesic flow on finite volume hyperbolic surfaces \cite{hopf}, the set ${\mathcal E}_0$ is full measure in $[\frac{\pi}{3}, \frac{2\pi}{3}]$. In fact, the set $[\frac{\pi}{3}, \frac{2\pi}{3}] \smallsetminus {\mathcal E}_0$ has Hausdorff dimension smaller than one;
see the discussion in \cite[Proof of Proposition 6]{hubert weiss}.

We define ${\mathcal E}$ to be the orbit of ${\mathcal E}_0$ in $S^1$ under the rotation group of order six, namely,
$$\textstyle {\mathcal E}={\mathcal E}_0 \cup (\frac{\pi}{3}+{\mathcal E}_0)\cup (\frac{2\pi}{3}+{\mathcal E}_0)\cup (\pi+{\mathcal E}_0)\cup (\frac{4\pi}{3}+{\mathcal E}_0)\cup (\frac{5\pi}{3}+{\mathcal E}_0).$$

\begin{theorem}[Ergodic directions]
\label{thm:main ergodic}
If $\theta \in {\mathcal E}$ then the flows $T_{\theta,+}$ and $T_{\theta,-}$ are ergodic when the domains $X_{\theta,+}$ and $X_{\theta,-}$ are equipped with their natural Lebesgue measures.
\end{theorem}
By remarks above this implies that the set of non-ergodic directions has Hausdorff dimension less than $1$.

\begin{example}[$\theta=\frac{\pi}{4}$]
The angle $\theta=\frac{\pi}{4}$ lies in ${\mathcal E}$.
To see this observe that $\theta'=\theta+\frac{\pi}{3}=\frac{7 \pi}{12} \in [\frac{\pi}{3},\frac{2\pi}{3}]$.
By definition $\vec{u}_{\theta'}$ is the unit tangent vector based at $i$ pointed into $\Delta$ at angle $\frac{5 \pi}{6}$ from the vertical. The billiard trajectory is periodic and is depicted in Figure \ref{fig:dark-triangles}. This billiard trajectory repeatedly travels above the line where $y=\frac{1}{\sqrt{3}}$ and so by the Theorem above the flows $T_{\theta,+}$ and $T_{\theta,-}$ are ergodic. A trajectory of $T_{\theta,+}$ is shown on the left side of Figure \ref{fig:dark-triangles}. We do not know if this trajectory equidistributes.
\end{example}

Now we will consider what happens when $\theta \not \in {\mathcal E}$. A special collection of such directions are those parallel to a vector in the {\em Eisenstein lattice}, the subgroup $\Lambda \subset \R^2$ redundantly generated by
\begin{equation}
\label{eq:e}
\textstyle \e_0=(1,0), \qquad
\e_1=(-\frac{1}{2},\frac{\sqrt{3}}{2}) \qquad \text{and} \qquad
\e_2=(-\frac{1}{2},-\frac{\sqrt{3}}{2}).
\end{equation}
To explain the dynamics in these directions we need some definitions.

\begin{figure}
\begin{center}
\includegraphics[height=2.5in]{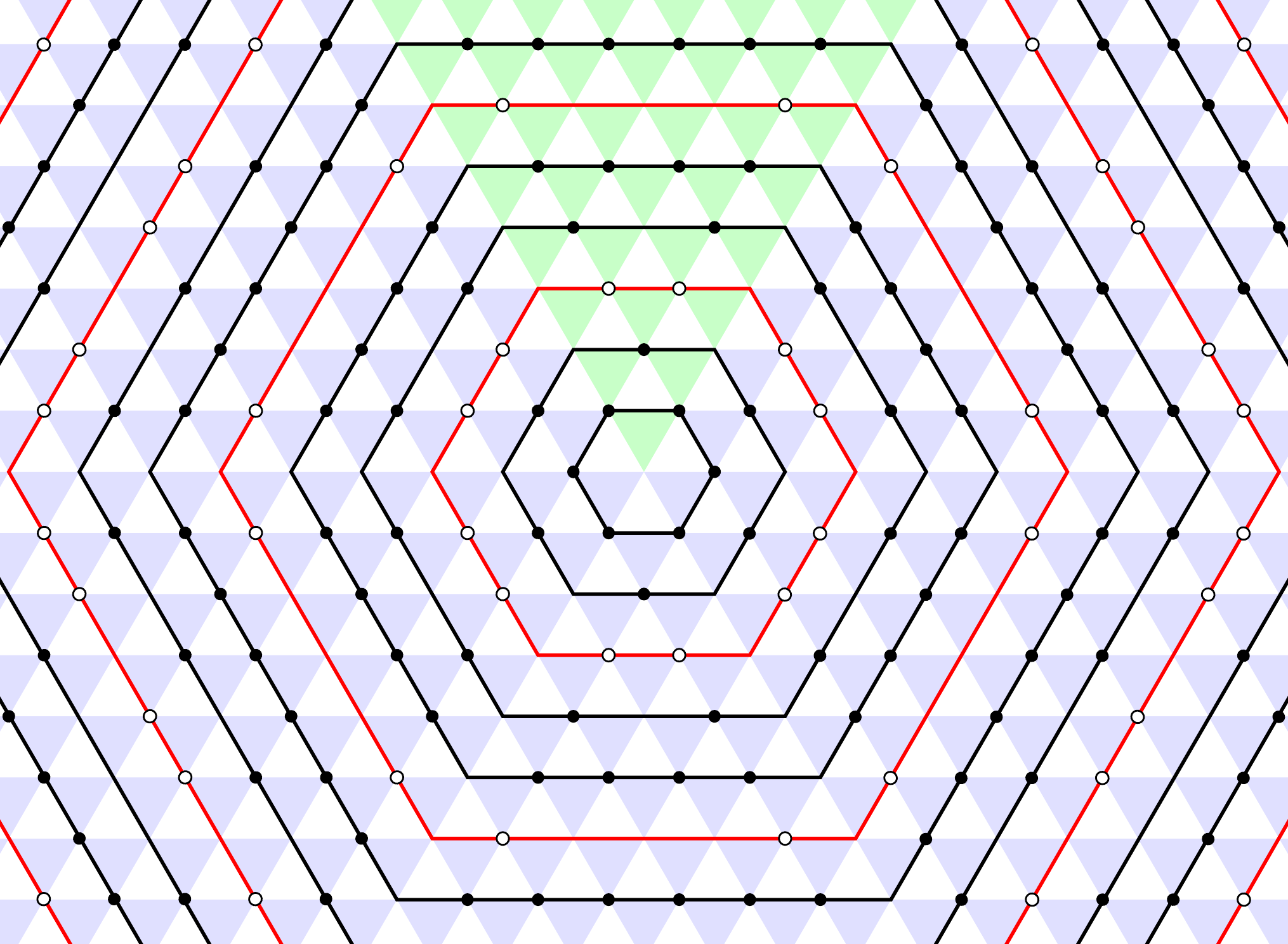}
\caption{A portion of the plane showing vectors associated to periodic (solid) and drift-periodic (open) directions. Lattice points without markers are not visible from the origin. The green region contains
vectors in $\Lambda$ with directions in $[\frac{\pi}{3},\frac{2\pi}{3}]$.}
\label{fig:concentric-hexagons}
\end{center}
\end{figure}

\begin{definition}
\label{def:Lambda vis}
We define $\Lambda_\vis \subset \Lambda$ to consist of those lattice points \emph{visible} from the origin, in the sense that they are not blocked by any other point of the lattice. Formally,
\begin{equation}
\label{eq:Lambda vis}
\Lambda_\vis = \Big\{\w \in \Lambda:~
\text{$c \w \not \in \Lambda$ for all $c \in \R$ with $0<c<1$}\}.
\end{equation}
These points are indicated by closed and open disks in Figure \ref{fig:concentric-hexagons}.
\end{definition}

\begin{definition}
For $\v\in\R^2$, let $\| \v \|_{\scriptsize\hexagon} = \text{min}\{|a|+|b|+|c| : \v = a\v_0+b\v_1+c\v_2\}$.
\end{definition}

This norm takes integer values on $\Lambda$: $\| \v \|_{\scriptsize\hexagon}=n$ when $\v$ is in the $n$th concentric hexagon shown in Figure \ref{fig:concentric-hexagons}.

We define two subsets of the circle:
\begin{align*}
{\mathcal D} & = \left\{\frac{\w}{|\w|} \in S^1:~
\text{$\w \in  \Lambda_\vis$ \quad and \quad $\| \w \|_{\scriptsize\hexagon} \equiv 0 \pmod{3}$}\right\}, \quad \text{and} \\
{\mathcal P} & = \left\{\frac{\w}{|\w|} \in S^1:~
\text{$\w \in  \Lambda_\vis$ \quad and \quad $\| \w \|_{\scriptsize\hexagon} \not \equiv 0 \pmod{3}$}\right\}. &
\end{align*}

\begin{theorem}[Lattice directions]
\label{thm:lattice}~
\begin{enumerate}
\item If $\theta \in {\mathcal P}$ then $T_{\theta,+}$ and $T_{\theta,-}$ are {\em completely periodic}, i.e., every non-singular trajectory is periodic. Conversely, every periodic trajectory is contained in $X_{\theta,+}$ or $X_{\theta,-}$ for some $\theta \in {\mathcal P}$.
\item If $\theta \in {\mathcal D}$ then $T_{\theta,+}$ and $T_{\theta,-}$ are {\em completely drift-periodic}, i.e., every non-singular trajectory is invariant under a non-trivial translational symmetry of the tiling. Conversely, every drift-periodic trajectory is contained in $X_{\theta,+}$ or $X_{\theta,-}$ for some $\theta \in {\mathcal D}$.
\end{enumerate}
\end{theorem}

We also show that periodic trajectories are preserved by order three rotation symmetries of the tiling (Corollary \ref{cor:periodic-order-three}),
and drift-periodic trajectories are invariant under one of the six non-trivial translations of the tiling that minimize translation distance (Corollary \ref{cor:drift-periodic-shift}).

\begin{remark}
We have $\theta \in {\mathcal P} \cup {\mathcal D}$ and
$\theta \in \left[\frac{\pi}{3}, \frac{2\pi}{3}\right]$
if and only if the billiard trajectory $g_t(\vec{u}_\theta)$ limits on one of the two ideal vertices of $\Delta$. Furthermore, if the trajectory limits on the ideal vertex at $\infty$ then $\theta \in {\mathcal D}$,
and if it limits on the ideal vertex at $\frac{\sqrt{3}}{3}$ then $\theta \in {\mathcal P}$. This together with order six rotational invariance determines the sets
${\mathcal P}$ and  ${\mathcal D}$.
\end{remark}

It is natural to ask what can be said about all trajectories since we have not covered all directions,
and ergodicity only says something about almost every trajectory in a direction. To this end we show:

\begin{theorem}
\label{thm:bounded implies periodic}
\begin{enumerate}
\item[(a)] All non-singular bounded trajectories of $T$ are periodic.
\item[(b)] If $x \in X$ has a non-singular trajectory under $T$, then the linear drift rate
$\lim_{t \to +\infty} \frac{|T^t(x)|}{t}$ is zero unless $x$ has a drift-periodic orbit (in which case this rate converges to a positive constant). Here $|T^t(x)|$ denotes the distance from the unit tangent vector of the basepoint of $T^t(x) \in X$ to the origin.
\end{enumerate}
\end{theorem}
Statement (a) of this theorem is proved at the end of \S \ref{sect: to a translation surface}.

In addition, we remark that in a set of directions of Hausdorff dimension more than $\frac{1}{2}$, the locally-finite ergodic invariant measures for $T_{\theta,+}$ and $T_{\theta,-}$ are classified: they are Maharam measures and are in bijection with group homomorphisms $\Z^2 \to \R_\times$. This follows from work in \cite{hooper}. 

Unfortunately, there are directions for which none of the results mentioned here apply. For example, if $\theta$ is parallel to $(\sqrt{2},3)$, then the 
trajectory $g_t(\vec{u}_\theta)$ is asymptotic to a periodic billiard trajectory below the line $y=\frac{1}{\sqrt{3}}$.

\begin{question}
Is it true that if $\theta$ is not parallel to a vector in the Eisenstein lattice, then the Lebesgue measure is ergodic for each of the flows $T_{\theta,+}$ and $T_{\theta,-}$?
\end{question}

Assuming an affirmative answer to this question, the system exhibits behavior very much analogous to the straight-line flow on a compact translation surface with the {\em lattice property}. By definition such a surface is stabilized by a lattice $\Gamma \subset \PGL(2,\R)$ acting by deformations of the translation surface structure.
In this setting, {\it Veech dichotomy} guarantees that the straight-line flow in any fixed direction on such a surface
is either uniquely ergodic or completely periodic with the later case corresponding to directions in which geodesics in $\H^2/\Gamma$ exit a cusp \cite[Theorem 8.2]{veech}.

In fact, lattice surfaces are essential to our proofs. From the tiling, we construct a translation surface $S$ that is a infinite cover of a torus; see \S \ref{sect: to a translation surface}. The flows
$T_{\theta,+}$ and $T_{\theta,-}$ are orbit equivalent to straight-line flows in some direction on $S$ (Theorem \ref{thm:orbit equivalence}).
This infinite translation surface has the lattice property (Proposition \ref{prop:generators}) and indeed the associated lattice in $\PGL(2,\R)$ is the triangle group obtained by reflections in the sides of our triangle $\Delta \subset \H^2$. We use the orbit equivalence and the symmetries provided by the Veech group to deduce Theorem \ref{thm:lattice}; see \S \ref{periodic directions}. The orbit-equivalence reduces the statement of Theorem \ref{thm:main ergodic} to a statement about ergodicity of straight-line flows on $S$. To verify ergodicity here we use a criterion due to Hubert and Weiss \cite{hubert weiss} developed into a context closer to ours by Artigiani \cite{artigiani} which provides a criterion for ergodicity of the straight-line flow on $S$. In \S \ref{sect:hubert weiss}, we offer an improvement to the constants in their argument and spell out a geometric description of the directions shown to be ergodic.
(The improvement of constants enabled us to decrease the value in \eqref{eq:E0} to $\frac{1}{\sqrt{3}}$ from $\frac{2}{\sqrt{3}}$. Our geometric description shows that Hubert and Weiss' notion of a direction being well-approximated by strips is equivalent in the lattice case to the corresponding geodesic in the Teichm\"uller curve having an accumulation point in an explicit finite union of cusp neighborhoods.)
We apply these methods in \S \ref{sect:ergodicity} where we prove Theorem \ref{thm:main ergodic}. Ergodicity of almost every direction on $S$ also follows from work of Ralston and Troubetzkoy \cite{ralston2} whose approach to these problems is similar to that of Hubert and Weiss.

Up to an affine change of coordinates the surface $S$ is square tiled: $S$ is an infinite cover regular of a flat torus branched at one point. This means that the straight-line flow on $S$ can be understood as a lift of the straight-line flow on a flat torus. From the above paragraph, this means that there is a section of each flow $T_{\theta,s}$ so that the return map to the section is a skew-product extension of an irrational rotation. When $\theta \in {\mathcal P} \cup {\mathcal D}$ the base dynamics are given by a rational rotation and otherwise the base dynamics are given by an irrational rotation. This is why trajectories on this tiling are unstable under a small change of direction, as noted  in \cite[\S 6]{icerm}. It is worth pointing out that in the context of straight-line flows on such infinite covers of tori, sometimes ergodicity is prevalent as here (e.g.  \cite{hooper hubert weiss}, \cite{hubert weiss} and  \cite{ralston2}) and in contrast sometimes ergodicity is atypical \cite{fu2}. It is not yet completely understood which infinite covers of a square torus exhibit ergodicity in almost every direction.

\section*{Acknowledgments}

Sergei Tabachnikov introduced us to tiling billiards in the context of Summer@ICERM, where two groups of students were the first to work on it \cite{icerm,2012}. ICERM  in $2012$  and Williams College in $2016$ provided excellent working environments.

We became interested in the trihexagonal tiling and observed phenomena described above by experimenting with the using computer programs written by the second author and Alex St Laurent. The programs model tiling billiards.
The program of St Laurent is publicly available \cite{st laurent} and can be run in a modern browser. Many
figures are generated using Sage (\url{http://www.sagemath.org/}) using the open-source FlatSurf package \cite{flatsurf} written by the second author and Vincent Delecroix. (Additional contributors to FlatSurf are welcome.)

We are grateful to Barak Weiss for helpful conversations related to the ergodicity criterion in \cite{hubert weiss}. Collaboration between the second author and Weiss was supported by BSF Grant 2016256.

Northwestern University provided travel funds for our continued collaboration. Contributions of the second author are based upon work supported by the National Science Foundation under Grant Number DMS-1500965 as well as a PSC-CUNY Award (funded
by The Professional Staff Congress and The City University of New
York).

\subsection{Context and questions}
\label{sect:context}

\subsubsection{Tiling billiards}
The connection between metamaterials with a negative index of refraction and the problem on planar tilings was made by Mascarenhas and Fluegel \cite{bloch}. Davis, DiPietro, Rustad and St Laurent named the system \emph{tiling billiards} and explored several special cases of the system, including triangle tilings and the trihexagonal tiling \cite{icerm}. They found examples of periodic trajectories in the trihexagonal tiling, constructed families of drift-periodic trajectories, and conjectured that dense trajectories and non-periodic escaping trajectories exist  (\cite{icerm}, Conjectures 5.12-5.13).


Concurrently with our work on the trihexagonal tiling, the first author with Baird-Smith, Fromm and Iyer in \cite{small} studied tiling billiards on triangle tilings, showing that trajectories on these tilings can be described by interval and polygon exchange transformations, and resolving additional conjectures from \cite{icerm}. These systems are quite different: if a trajectory visits a single tile twice then the trajectory is periodic.

One can interpolate between the trihexagonal tiling and the tiling by equilateral triangles by simultaneously shrinking the edges of downward-pointing triangles until the downward triangles disappear. See Figure \ref{fig:deformed tiling}.

\afterpage{
\begin{figure}
\begin{center}
\includegraphics[width=0.95\textwidth]{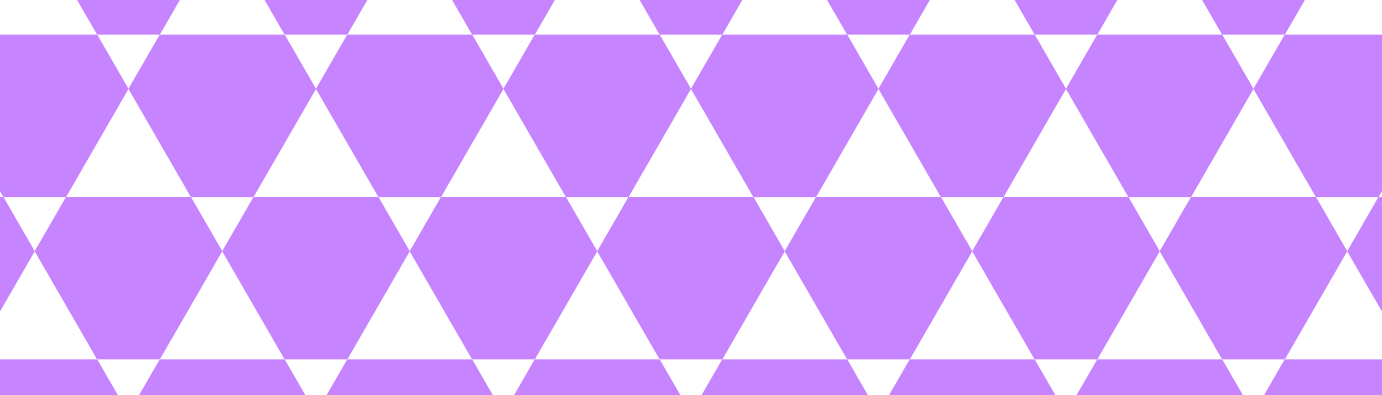}
\caption{A deformed trihexagonal tiling where the edges of downward pointing triangles have half the length of edges of upward pointing triangles.}
\label{fig:deformed tiling}
\end{center}
\end{figure}}

\begin{question}
How does the tiling dynamics change as we vary ratio of lengths of edges of the upward and downward pointing triangles?
\end{question}

More generally, we wish to understand the dynamics of trajectories on tilings, e.g.:

\begin{question} What feature of a tiling makes it possible (as here) or impossible (as in triangle tilings) to have dense regions in tiling billiard trajectories in the plane?
\end{question}

\subsubsection{Periodic billiard tables and related systems}

Others have studied various billiard systems on a periodically-tiled plane. In the {\it wind-tree model}, there are axis-parallel rectangular obstacles (trees) at lattice points, with a billiard flow (wind) in the plane outside of the obstacles. Delecroix, Hubert and Leli\`evre showed that no matter the size of the obstacles, for almost every direction the polynomial diffusion rate is $2/3$ \cite{dhl}. Subsequently Delecroix and Zorich determined the diffusion rates for other periodic families of objects with axis-parallel edges, such as the wind-tree model with a periodic set of obstacles removed, or obstacles with a more complicated shape \cite{dz}. Other work on the wind-tree model is in \cite{ah,delecroix,hlt}, and on other periodic billiard tables in \cite{fu1,fu2}.

Other billiard systems are also motivated by optics.
Fr\k{a}czek and Schmoll \cite{fs}, Fr\k{a}czek, Shi and Ulcigrai \cite{fsu}, and Artigiani \cite{artigiani} studied the plane with periodic optical obstacles called {\it Eaton lenses}, which act as a perfect optical retro-flector: when a light ray enters, it exits parallel but traveling in the opposite direction.
In each paper, the authors replaced the spherical lenses with slits in the plane, and constructed a related translation surface. Typically the associated flows are non-ergodic \cite{fs} but Artigiani \cite{artigiani} demonstrated that many configurations lead to ergodic flows.

\subsection{Outline of paper}
\begin{itemize}
\item In $\S$\ref{basic-definitions}, we introduce the tiling billiards system, describe the \emph{folding technique}, and give several fundamental results specific to the trihexagonal tiling.
\item In $\S$ \ref{sect: to a translation surface}, we define a translation surface $S$ from the tiling, and state and prove the orbit-equivalence result.
\item In $\S$ \ref{hidden symmetries}, we give specific results about $S$. In particular, we find the Veech (Affine symmetry) group of the surface.
\item In $\S$ \ref{periodic directions}, we use the symmetries of $S$ to investigate periodic and drift-periodic directions on $S$. We use the orbit-equivalence to describe the periodic and drift-periodic directions for the tiling flow.
\item In \S \ref{sect:hubert weiss}, we revisit the well-approximation by strips criterion for ergodicity due to Hubert and Weiss.
\item In $\S$ \ref{sect:ergodicity}, we prove the ergodicity of almost every aperiodic direction.
\end{itemize}

\section{Definitions and basic observations}\label{basic-definitions}
\label{sect:definitions}

\subsection{The billiard flow on a tiling}
Consider a tiling ${\mathcal T}$ of the plane by regions with piecewise $C^1$ boundaries. For concreteness denote these regions by $R_i$. We consider a ray in some region $R_1$ in ${\mathcal T}$, which shares a boundary with some region $R_2$. When the ray intersects the boundary between $R_1$ and $R_2$, it is reflected across the tangent line to the boundary curve at the point of intersection. (If the regions are polygonal, as they are here, the ray is reflected across the edge itself.) We can extend this new ray to a line, and continue along this line in the traveling away from the intersection point.
The {\it billiard flow} on ${\mathcal T}$ is the flow defined by refracting in this way all trajectories across the boundaries they hit. See Figure \ref{fig:trajectory}.
In the case that the tiling can be $2$-colored, this agrees with the flow of light when the tiles of each color are composed with transparent materials with equal but opposite indices of refraction and we also call the billiard flow the {\em refractive flow}.

\subsection{The folding construction}
Because the billiard flow reflects a trajectory across each edge of the tiling, we can use {\it folding} to significantly simplify our analysis, as follows. When a trajectory crosses an edge of the tiling (the left side of Figure \ref{fig:folding}), we fold the tiling across that edge, the dotted line in Figure \ref{fig:folding}. When the trajectory crosses the next edge, we fold across that edge as well (the middle of Figure \ref{fig:folding}). The result is that the trajectory on the folded tiling is always along a single line, alternating at each edge crossing between forward and backward (the right side of Figure \ref{fig:folding}).

\begin{figure}
\begin{center}
\includegraphics[width=5in]{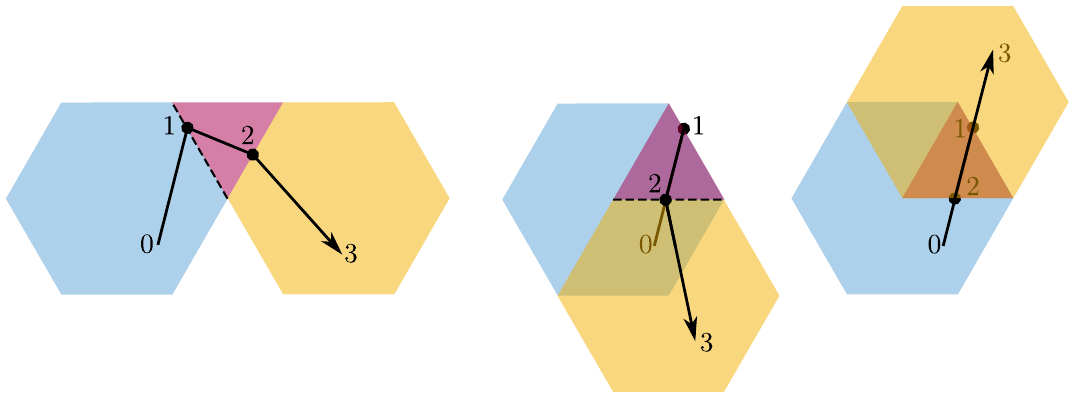}
\caption{When we fold the tiling along each edge that the trajectory crosses, the folded trajectory goes back and forth along a single line.}
\label{fig:folding}
\end{center}
\end{figure}

\subsection{Applications to the trihexagonal tiling}
\label{sect:applications}
\begin{lemma}\label{lem:three-angles}
In the trihexagonal tiling, a billiard trajectory initially traveling in direction $\theta$ in a hexagon (resp. triangle)
is traveling in a direction in the set $\{\theta, \theta+2\pi/3, \theta+4\pi/3\}$ whenever it returns to a hexagon (resp. triangle).
\end{lemma}

\begin{proof}
When a trajectory crosses an edge, its direction is transformed via a reflection across the line between the midpoint of that edge and the center of the polygon (hexagon or triangle). These reflections form a symmetry group of order $6$. Furthermore, to get back into the polygon of the same kind, an even number of reflections is required, since the trajectory alternates between triangles and hexagons. Thus the change in direction between returns to hexagons (or triangles) is by the action of an element of the group of rotations of order three.
\end{proof}

A hexagon and a triangle meet at every edge of the trihexagonal tiling, so the folded trajectory goes ``forward'' in hexagons and ``backward'' in triangles (or vice-versa, depending on convention), as in the right side of Figure \ref{fig:folding}. Since hexagons are larger than triangles, the trajectory makes forward progress in the folded tiling. This is in contrast to the behavior on the square tiling, for example, where every trajectory folds up to a finite line segment, on which it goes back and forth (see \cite{icerm}, Figure 5).

\begin{lemma}\label{lem:singular-center}
The center of each hexagon is singular, in the sense that a refractive flow through the center always hits a singularity.
\end{lemma}

\begin{proof}
A trajectory through the center of the hexagon then passes through a triangle, since every edge of a hexagon is shared with a triangle. When we fold the trajectory along this edge, the third vertex of the triangle folds down to the center of the hexagon. Thus every trajectory through the center of a hexagon passes through a vertex of a triangle, and is singular.
\end{proof}


\begin{lemma}
The clockwise or counter-clockwise travel around centers of hexagons is invariant under the refractive flow.
\end{lemma}

\begin{proof}
Suppose, without loss of generality, that a given trajectory travels counter-clockwise around a particular hexagon, towards the top horizontal edge (the left side of Figure \ref{fig:orientation}). The refractive flow is symmetric across this edge, and the trajectory stays on the right side of the lower dotted line in Figure \ref{fig:orientation}, so it must also stay on the right of the upper dotted line. Thus it hits the right edge of the triangle and passes into the hexagon on the right. When we fold across these two edges, superimposing the second hexagon on the first, we can see that the trajectory travels counter-clockwise in the second hexagon as well. Since counter-clockwise travel in one hexagon leads to counter-clockwise travel in the next hexagon, and the same holds for clockwise travel, the orientation is invariant under the refractive flow.
\end{proof}

\begin{figure}
\begin{center}
\includegraphics[width=3in]{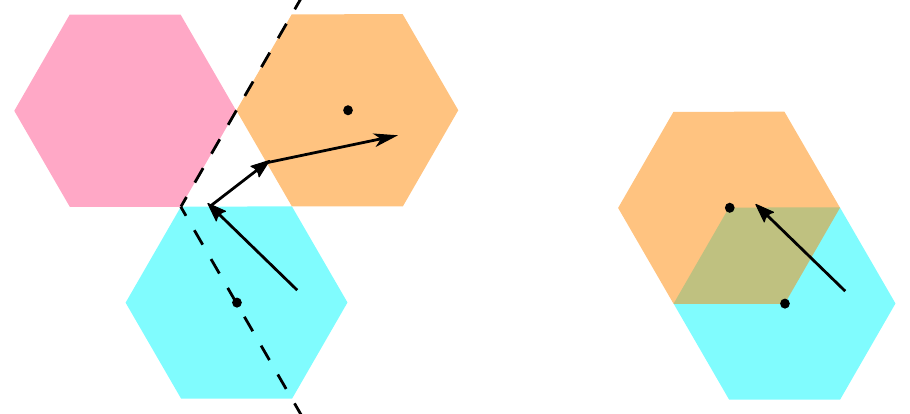}
\caption{Counter-clockwise travel in one hexagon leads to counter-clockwise travel in the next hexagon as well.}
\label{fig:orientation}
\end{center}
\end{figure}



It therefore makes sense to distinguish trajectories by the direction they travel around the centers of hexagons. We introduce notation for this:

\begin{definition}[Restricted refractive flows]
We define $T_{\theta,+}$ (resp. $T_{\theta,-}$) to be the refractive flow $T$ restricted to the set
$X_{\theta,+} \subset X$ (resp. $X_{\theta,-} \subset X$)
consisting of unit tangent vectors representing initial positions and directions of trajectories that travel in a direction in the set
 $\{\theta, \theta+2\pi/3, \theta+4\pi/3\}$ when within a hexagon and travel counter-clockwise (resp. clockwise) around the centers of hexagons.
\end{definition}

\begin{definition}[Lebesgue measure on $X_{\theta,s}$]
Consider the
the identification between $X$ and $\R^2 \times S^1$ which recovers a unit tangent vector's basepoint and direction.
By Lemma \ref{lem:three-angles}, vectors in $X_{\theta,s}$ point in one of six directions (for most $\theta$ or in three directions if $\theta$ is perpendicular to the edges of the tiles). We define $A \subset X_{\theta,s}$ to be {\em Lebesgue measurable} if for each of these six (resp. three) directions, $\theta'$, the set
$$A_{\theta'}=\{\v \in \R^2:~(\v,\theta') \in A\}$$
is Lebesgue measurable as a subset of $\R^2$. We define the Lebesgue measure of $A$ to be the sum of the Lebesgue measures of the six (resp. three) sets $A_{\theta'}$.
\end{definition}
Because the index of refraction is negative one, it follows that Lebesgue measure on $X_{\theta,s}$ is $T_{\theta,s}$-invariant. We leave the details to the reader.

We will find it convenient to use the symmetries of our tiling to limit the flows defined above we need to consider.
It turns out that up to the symmetries of the tiling,
all these flows $T_{\theta,s}$ are conjugate to ones which travel in a direction in the interval $[\frac{\pi}{3},\frac{2 \pi}{3}]$ and that travel counter-clockwise around the centers of hexagons. Formally:

\begin{proposition}
\label{prop: standardization}
For any $\theta \in S^1$ and any sign $s \in \{+,-\}$ there
is a $\theta' \in [\frac{\pi}{3},\frac{2 \pi}{3}]$ and an
isometry of the tiling $I$ so that
$I(X_{\theta,s})=X_{\theta',+}$ and
$T_{\theta',+}^t=I \circ T^t_{\theta,s} \circ I^{-1}$.
\end{proposition}
\begin{proof}
Applying a reflective symmetry of the tiling swaps trajectories that travel clockwise around the centers of hexagons with ones that travel counter-clockwise. The interval $[\frac{\pi}{3},\frac{2 \pi}{3}]$ represents a fundamental domain for the action of the rotational symmetries of the tiling on the circle of directions.
So, we can always take $I$ to be such a rotational symmetry or a composition of a reflective symmetry and a rotational symmetry.
\end{proof}

The above Proposition allows us to assume that our trajectories travel in a direction $\theta\in [\frac{\pi}{3},\frac{2 \pi}{3}]$ and travel counter-clockwise around centers of hexagons.

Let $\pi:X \to \R^2$ denote the projection of a unit tangent vector in the plane to the vector's basepoint in the plane. The next few results describe the closure of the projection of $X_{\theta,s} \subset X$.

\begin{lemma}\label{lem:centers}
If $\frac{\pi}{3}\leq \theta < \frac{\pi}{2}$
(respectively, $\frac{\pi}{2}< \theta \leq \frac{2\pi}{3}$), then $\overline{\pi(X_{\theta,+})}$ is all of the plane but the
periodic family open triangles in upward (resp. downward) pointing triangles in the tiling which are bounded by segments of singular trajectories hitting vertices of the triangle from the tiling in forward (resp. backward) time. See Figure \ref{fig:missing triangles} for illustration of these missing triangles. When $\theta=\frac{\pi}{2}$ we have $\overline{\pi(X_{\theta,+})}=\R^2$.
\end{lemma}

\begin{figure}
\begin{center}
\includegraphics[width=0.45\textwidth]{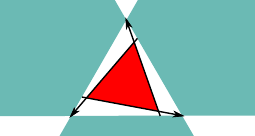} \hfill
\includegraphics[width=0.45\textwidth]{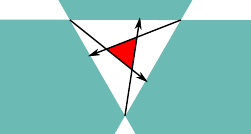}
\caption{Left: The red triangle is not in $\overline{\pi(X_{\theta,+})}$ when
$\frac{\pi}{3}\leq \theta < \frac{\pi}{2}$.
Right: Similar but shown when $\frac{\pi}{2}\leq \theta < \frac{2\pi}{3}$.}
\label{fig:missing triangles}
\end{center}
\end{figure}

Note that it is a direct consequence that no trajectory is dense in the plane.

\begin{corollary}
For each non-singular trajectory, there is a periodic family of open sets containing centers of either the upward-pointing triangles or downward-pointing triangles, such that the trajectory never enters the family of sets.
\end{corollary}

See Figure \ref{fig:dark-triangles} for an example.

\begin{proof}
Each trajectory lies in some $X_{\theta',s}$, which up to a tiling symmetry has the form $X_{\theta,+}$ for some $\theta \in \left[\frac{\pi}{3},\frac{2\pi}{3}\right]$ by Proposition \ref{prop: standardization}. If $\theta \neq \frac{\pi}{2}$, then this follows directly from Lemma \ref{lem:centers}. When $\theta=\frac{\pi}{2}$, it is not hard to show that all trajectories are periodic and singular trajectories can not intersect triangle centers.
\end{proof}

\begin{proof}[Proof of Lemma \ref{lem:centers}]
The lemma concerns the case when $\frac{\pi}{3}\leq \theta \leq \frac{2\pi}{3}$. We will leave the case
of $\theta=\frac{\pi}{2}$ to the reader though it follows from the same type of observations.

First, we will show that the flow covers all of each hexagon. The counter-clockwise flow in direction $\theta$ covers half of the hexagon, everything to the right of the singular trajectory through the hexagon center. The flow in direction $\theta$ also includes the flows in directions $\theta+\frac{2\pi}3$ and $\theta+\frac{4\pi}3$. Thus we also cover the images of this half-hexagon under rotations of order three.
The three rotated images of the half hexagon cover the entire hexagon.

Now we will explain how this missing triangles appear.
A counter-clockwise trajectory traveling in a direction $\theta \in [\frac{\pi}{3},\frac{\pi}{2})$  in a hexagon misses the centers of upward-pointing triangles: Consider a flow in this direction, which is to the right of vertical, on a hexagon (Figure \ref{fig:missing-centers}a). Recall that the flow through the center of the hexagon is singular (Lemma \ref{lem:singular-center}), and passes through the top vertex of the next triangle. This singular flow (thick line) divides the flow on the hexagon and triangle into a left side (clockwise flow) and a right side (counter-clockwise flow). Since we restrict our attention to counter-clockwise flow, only the flow to the right of the singularity is allowed, which misses the center of the triangle. Trajectories
in $X_{\theta,+}$ also travel within hexagons in directions which differ from $\theta$ by a rotation
of order three (i.e., by rotations of $\pm \frac{2\pi}{3}$); see Lemma \ref{lem:three-angles}. Observe that there are isometries of our
tiling whose derivatives realize these rotations. Such isometries preserve the sets of upward (resp. downward) pointing triangles. It follows that a triangular island in the center of each upward pointing triangle is missed, whose boundaries are the singular flows through the centers of adjacent hexagons.

By the same argument, a counter-clockwise trajectory with $\pi/2<\theta\leq 2\pi/3$ misses a triangular island at the centers of the downward-pointing triangles (Figure \ref{fig:missing-centers}b).
%

Now we will show that when $\theta\in[\pi/3,\pi/2)$,
the image of $X_{\theta,+}$ covers all of the downward-facing triangles (Figure \ref{fig:covers-triangle}a). The singular trajectory in this direction through the hexagon center intersects the edge between the previous (downward-facing) triangle and the hexagon to the left of its midpoint. The flow is on the right side of the singular trajectory, so this flow covers a portion of the triangle that includes more than half of its area, including the triangle center. The flow in direction $\theta$ also includes the three rotations of order $3$ of direction $\theta$, and the $3$ rotations of the portion of the triangle cover the entire triangle.

Again by the same argument, a counter-clockwise flow in direction $\theta\in(\pi/2,2\pi/3]$ covers all of the upward-facing triangles (Figure \ref{fig:covers-triangle}b).
\end{proof}


\begin{figure}
\begin{center}
\includegraphics[width=0.45\textwidth]{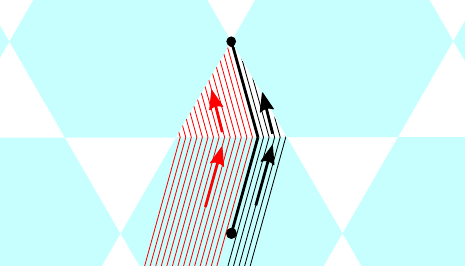} \hfill
\includegraphics[width=0.45\textwidth]{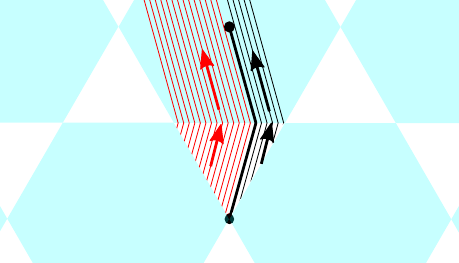}
\caption{Counter-clockwise trajectories are shown in black and clockwise in red.
Counter-clockwise trajectories that travel (a) to the right of vertical in hexagons miss the centers of upward triangles, while those that travel (b) to the left of vertical in hexagons miss the centers of downward triangles.}
\label{fig:missing-centers}
\end{center}
\end{figure}

\begin{remark}
Although the flows $T_{\theta,s}$ equipped with their Lebesgue measures are often ergodic (Theorem \ref{thm:main ergodic}), trajectories do not equidistribute in the plane (in the sense of the ratio Ergodic theorem) even outside the missed triangles,
because different points in the plane are hit different numbers of times (as few as zero and as many as three times) by the projection of $X_{\theta,s}$ to the plane. For example, generic trajectories run through some regions three times as often as other regions of the same area.
\end{remark}

\begin{figure}
\begin{center}
\includegraphics[width=0.45\textwidth]{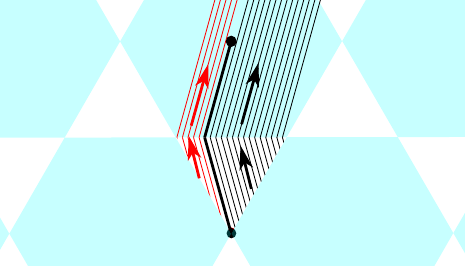} \hfill
\includegraphics[width=0.45\textwidth]{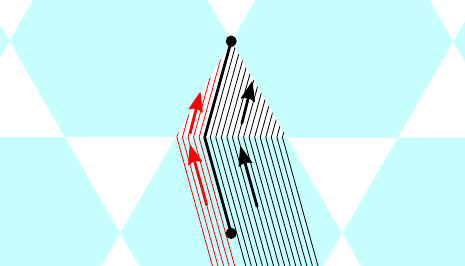}
\caption{Counter-clockwise trajectories that travel (a) to the right of vertical in hexagons cover all the downward triangles, while those that travel (b) to the left of vertical in hexagons cover all the upward triangles.}
\label{fig:covers-triangle}
\end{center}
\end{figure}

\begin{proof}[Proof of Theorem \ref{thm:1} assuming Theorem \ref{thm:main ergodic}]
Recall, we need to show that for Lebesgue-almost every starting point $\x$ and starting direction $\theta$ our $T$-trajectory is dense in the plane minus a periodic family of triangular islands.

Recall that the set of ergodic directions $\mathcal E$ used in Theorem \ref{thm:main ergodic} is full measure. (See the discussion under the theorem.) So Lebesgue-almost surely $\theta \in {\mathcal E}$ and $- \theta \in \mathcal E$. Therefore by Theorem \ref{thm:main ergodic}, we may assume that $T_{\theta,+}$, $T_{\theta,-}$, $T_{-\theta,+}$ and $T_{-\theta,-}$ are all ergodic.

Think of $\theta$ as fixed and satisfying this statement that the four maps above are ergodic. Then Lebesgue almost-every $\x \in \R^2$ yields a unit tangent vector $(\x,\theta)$ lying in one of four possible sets:
\begin{itemize}
\item It lies in $X_{\theta,+}$ (resp. $X_{\theta,-}$) if $\x$ lies in the interior of a hexagon and flow of $\x$ in direction $\theta$ moves counter-clockwise (resp. clockwise) around the center of the hexagon.
\item It lies in $X_{-\theta,+}$ (resp. $X_{-\theta,-}$) if $\x$ lies in the interior of a triangle and the refractive trajectory travels counter-clockwise (resp. clockwise) around the center of the next hexagon entered.
\end{itemize}
Since each of the four flows $T_{\theta,+}$, $T_{\theta,-}$, $T_{-\theta,+}$ and $T_{-\theta,-}$
are ergodic (and the natural topologies on the domains have a countable basis where all open sets are assigned positive mass by Lebesgue measure), it follows from standard results in ergodic theory that almost every trajectory in any of these four domains is dense in that domain.
\compat{Leaving a link as an informal reference for us: \url{http://math.stackexchange.com/questions/990229/ergodic-system-has-a-e-dense-orbits}.}
We conclude
that almost every $\x \in \R^2$ gives a unit tangent vector $(\x, \theta)$ whose trajectory is dense in one of the four
sets $X_{\theta,+}$, $X_{\theta,-}$, $X_{-\theta,+}$ or $X_{-\theta,-}$.
\end{proof}

\section{From the tiling to a translation surface}
\label{sect: to a translation surface}

\comref{I think that, in general, there is some confusion due to the fact that usually translation surfaces are obtained simply by using translations from
polygons in the plane. Here, you use rotations, but still obtain a translation surface. This implies that the directions are not “what you see is
what you get”, meaning that the directions on S are different from the
ones on the tiling of $\R^2$ by rhombi. This is stated for v i , but I think it
should be clearly written that a direction on the plane does not correspond to the same direction on S. Finally, when in Section 4.5 you use
right Dehn twists along cylinders in directions v i you could reference the
pictures you included of the cylinders to help make this distinction clear.}

%

We define the vectors $\e_0$, $\e_1$ and $\e_2$ as in \eqref{eq:e}.
We normalize our tiling so that centers of hexagons lie in $2\Lambda = \langle 2 \e_0, 2\e_1, 2 \e_2\rangle$.
This makes all edges of our tiling have length one. We use $H_{\c}$ to denote the hexagon whose center is $\c \in 2\Lambda$.
We split each hexagon into three rhombuses: $H_{\c}=\bigcup_{i=0}^2 R^i_{\c}$ where
$R^i_{\c}$ is the rhombus which has the hexagon's center $\c$ as one vertex and has edge vectors at this vertex given by the two vectors $-\e_{i+1}$ and $-\e_{i-1}$ (where subscripts $\e_\ast$
are interpreted modulo three). This means that $R^i_{\c}$ has vertices
$\c$, $\c-\e_{i+1}$, $\c+\v_i$, and $\c-\e_{i-1}$. Note that the vector representing the short diagonal of $R^i_\c$ pointing away from the center of the hexagon is $\v_i$.
We label the edges of each rhombus $R^i_\c$ by $\{1,2,3,4\}$ counterclockwise starting with the edge leaving the center $\c$ in direction $-\e_{i+1}$. See Figure \ref{fig:rhombi}.

\begin{figure}[b]
\begin{center}
\includegraphics[height=3.5in]{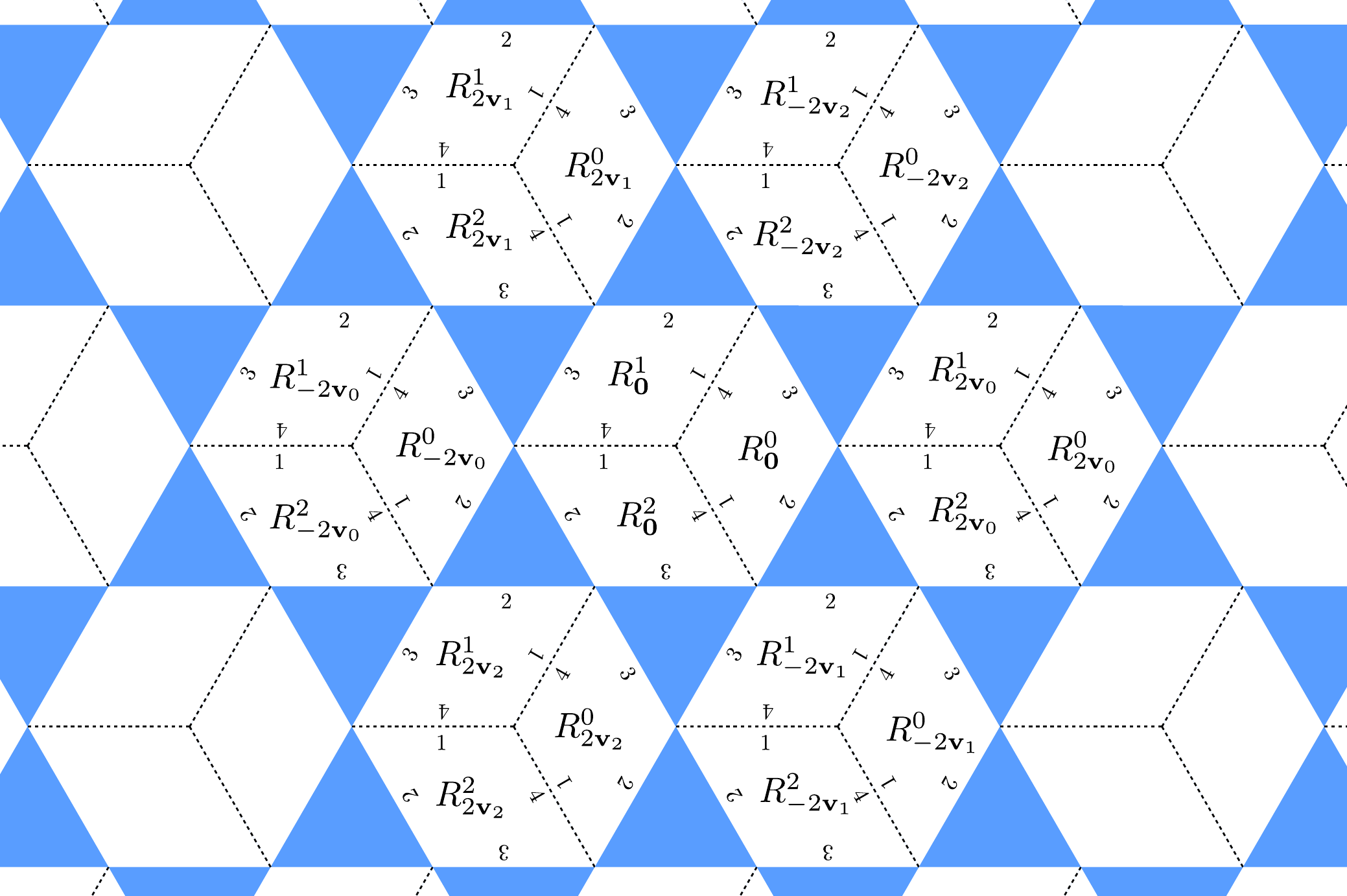}
\caption{Some named rhombi and their edge labels.}
\label{fig:rhombi}
\end{center}
\end{figure}

The decomposition into rhombi is useful for understanding the behaviors of trajectories:

\begin{lemma}[Trajectories exiting rhombi]
\comref{Lemma 3.1, page 13}
\compat{We decided the referee's remark was erroneous. See the email.}
Consider a trajectory traveling through the interior of a rhombus $R^i_\c$ in a direction which makes an angle $\theta \in [\frac{\pi}{3},\frac{2 \pi}{3}]$ with the vector $\e_i$. If such a trajectory exits through the interior of an edge, that edge must have label either $3$ or $4$. In addition:

\label{lem:exit}
\begin{enumerate}
\item Such a trajectory exiting through edge $3$ will exit hexagon $H_\c$
through that edge, pass through the triangle opposite that side, then enter hexagon
$H_{\c+2 \v_{i}}$ and move through the hexagon until it enters rhombus $R^{i-1}_{\c+2 \v_{i}}$ through edge $1$. Furthermore if the trajectory exits $R^i_\c$ at point $p$ of edge $3$ traveling in direction $\theta$, and enters $R^{i-1}_{\c+2 \v_{i}}$ through edge $1$ at point $q$ traveling in direction $\eta$, then
$\eta=\theta-\frac{2\pi}{3}$ and $q$ is the image of $p$ under the rotation by angle
$\frac{-2\pi}{3}$ about the vertex in common between the rhombi.
\item Such a trajectory exiting through edge $4$ will exit hexagon $H_\c$, pass through a triangle, and then enter hexagon $H_{\c-2 \v_{i-1}}$ through edge $2$
of rhombus $R^{i-1}_{\c-2 \v_{i-1}}$. Furthermore if the trajectory exits $R^i_\c$
at point $p$ in edge $4$ traveling in direction $\theta$, and enters $R^{i-1}_{\c-2 \v_{i-1}}$ at point $q$ of edge $2$ traveling in direction $\eta$, then
$\eta=\theta-\frac{2\pi}{3}$ and $q$ is the image of $p$ under the rotation by angle
$\frac{-2\pi}{3}$ about the vertex in common between the rhombi.
\end{enumerate}
\end{lemma}

\begin{proof} We will prove these for $\mathbf{c}=\mathbf{0}$ using Figure \ref{fig:rhombi}, and the general case follows by translation.
\begin{enumerate}
\item \label{exiting3}
Consider a trajectory exiting rhombus $R_{\mathbf{0}}^0$ through edge $3$, with a direction $\theta\in [\frac{\pi}{3},\frac{2 \pi}{3}]$. This trajectory crosses the adjacent triangle and then enters a hexagon. Because of the angle condition, it must enter the lower hexagon $H_{2\v_0}$, passing through rhombus $R_{2\v_0}^1$ and into rhombus $R_{2\v_0}^2$. We can see this by folding the tiling across the two edges that the trajectory crosses, so that $R_{\mathbf{0}}^0$ and $R_{2\v_0}^2$ are adjacent, as in the left two pictures of Figure \ref{fig:exiting}.

The two lines we fold along are separated by an angle of $\pi/3$, so folding along both of them amounts to a rotation by twice the angle, $2\pi/3$, about their intersection point. Thus, $q$ is the image of $p$ under the inverse of this rotation, a rotation by $-2\pi/3$ about the common vertex, and similarly the angle of the trajectory in $R_{2\v_0}^2$ is  $\eta = \theta - \frac{2\pi}3$.

The geometry is the same for $R_{\mathbf{0}}^1$ and $R_{\mathbf{0}}^2$, with the picture rotated by $2\pi/3$ and $4\pi/3$, respectively.

\item \label{exiting4}
Consider a trajectory exiting rhombus $R_{\mathbf{0}}^0$ through edge $4$, with a direction $\theta \in [\frac{\pi}{3},\frac{2 \pi}{3}]$. This trajectory crosses $R_{\mathbf{0}}^1$ and then enters an adjacent triangle. Because of the angle condition, it must exit through edge $2$ of $R_{\mathbf{0}}^1$ and enter the hexagon on the right $H_{-2\v_2}$, into rhombus $R_{-2\v_2}^2$. We can see this by folding the tiling across the two edges that the trajectory crosses, so that  $R_{\mathbf{0}}^0$ and $R_{-2\v_2}^2$ are adjacent, as in the right two pictures of Figure \ref{fig:exiting}.

The two lines we fold along are again separated by an angle of $\pi/3$, so the angle result and the locations of $p$ and $q$ follow as in part (\ref{exiting3}). Again, the geometry for $R_{\mathbf{0}}^1$ and $R_{\mathbf{0}}^2$ is the same, but with the picture rotated.
\end{enumerate}
To see the general case of $\c \neq \0$, observe that the refractive flow is invariant under translations by vectors in $2 \Lambda$. These translations act on rhombus labels by addition.
\end{proof}

\begin{figure}
\begin{center}
\includegraphics[height=1.2in]{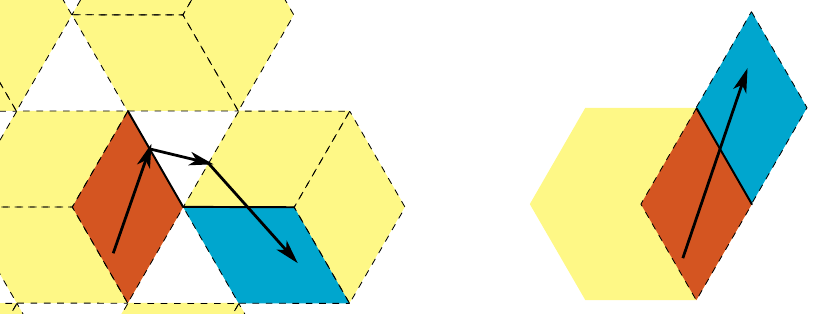}
\hfill
\includegraphics[height=1.2in]{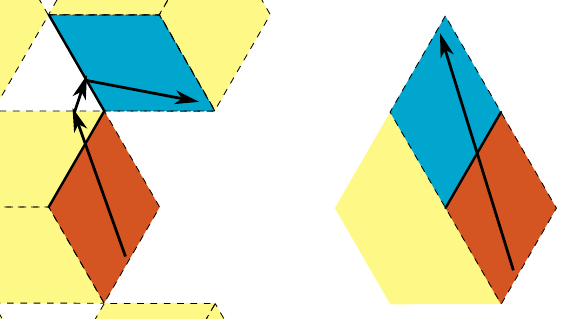}
\caption{A trajectory leaving through the solid edge of the orange rhombus enters the blue rhombus through the solid edge.}
\label{fig:exiting}
\end{center}
\end{figure}

\begin{figure}
\begin{center}
\includegraphics[width=4in]{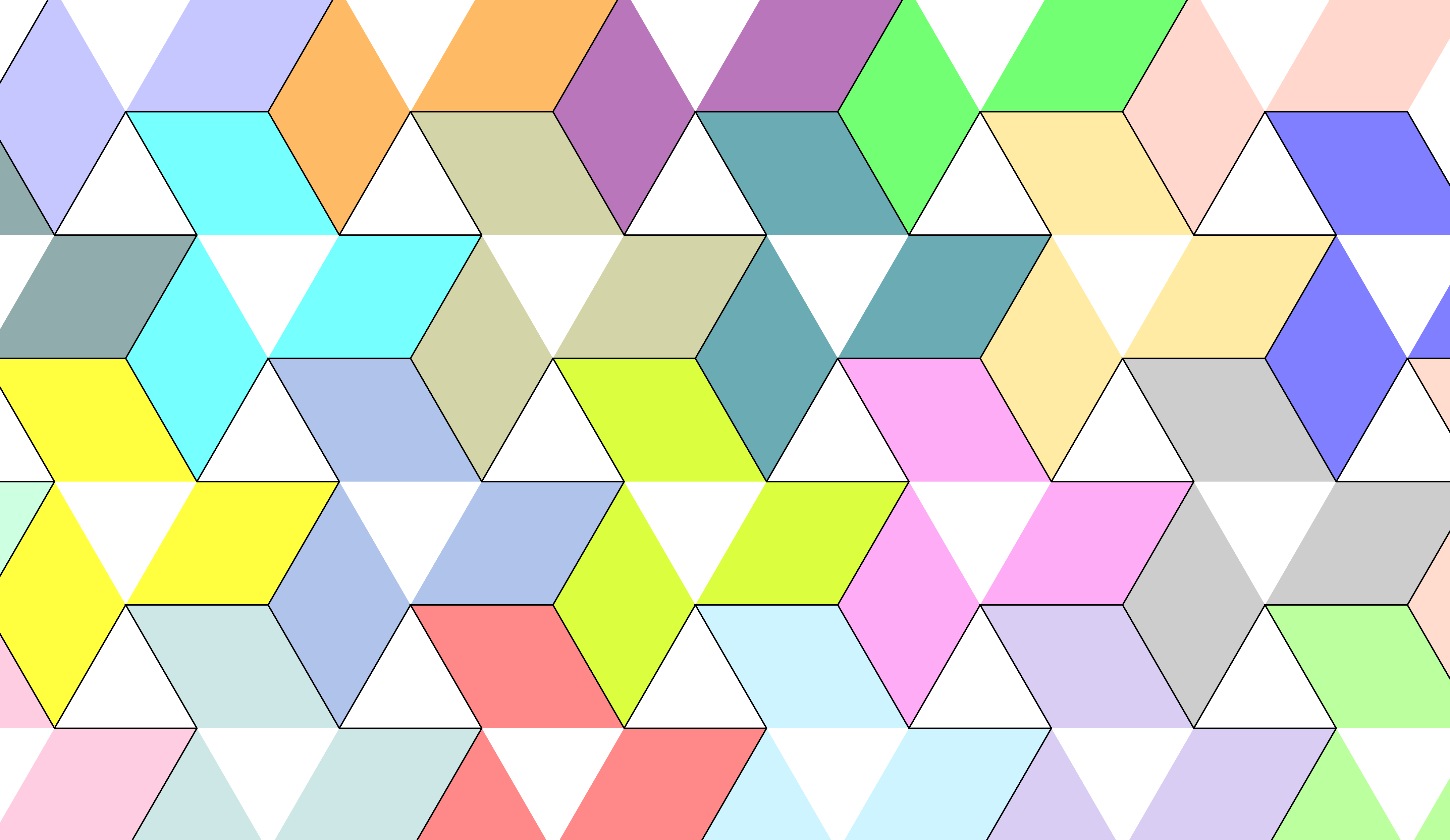}
\caption{A decomposition of the hexagons into subsets of three rhombi. Triples of rhombi surrounding downward pointing triangles are shown in the same color and outlined.}
\label{fig:pinwheel}
\end{center}
\end{figure}

We will build a metric surface $S$ by gluing the rhombi edge-to-edge by rotations.
Technically, to consider $S$ a surface, the vertices of the rhombi must be removed before making these identifications, because otherwise these points will not have neighborhoods homeomorphic to a neighborhood in the plane because infinitely many rhombi will be identified at a single vertex.
Our identifications turn out to always be by a rotation by $\pm \frac{2\pi}{3}$ about a vertex of the tiling. We partition the rhombi into triples which appear around downward-pointing triangles in the tiling (Figure \ref{fig:pinwheel}).
Within such a triple, we identify the rhombi along
edges $2$ and $4$ to form a cylinder. See the left side of Figure \ref{fig:cylinder}. Similarly, there is another partition of the rhombi into triples given by considering triples of rhombi surrounding upward-pointed triangles in the tiling. We glue edges labeled $1$ and $3$ of these rhombi together to form a cylinder.
See the right side of Figure \ref{fig:cylinder}. Formally, the gluings are described by the following rules, where we use
$E_j(R^i_\c)$ to denote the edge with label $j$ of rhombus $R^i_\c$:
\begin{equation}
\label{eq:gluings}
\begin{array}{rclcrcl}
E_1(R^i_\c) & \to & E_3(R^{i+1}_{\c-2 \v_{i+1}}),
&  &
E_3(R^i_\c) & \to & E_1(R^{i-1}_{\c+2 \v_{i}}),\\
E_2(R^i_\c) & \to & E_4(R^{i+1}_{\c+2 \v_{i}})
& \quad \text{and} \quad &
E_4(R^i_\c) & \to & E_2(R^{i-1}_{\c-2 \v_{i-1}}).
\end{array}
\end{equation}

\begin{figure}
\begin{center}
\includegraphics[width=.48\textwidth]{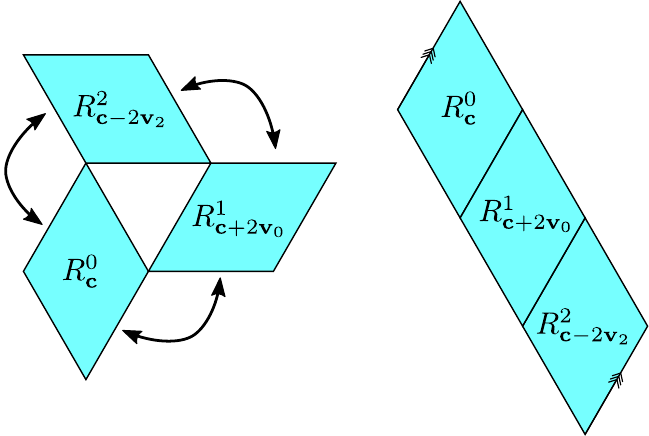}
\hfill
\includegraphics[width=.48\textwidth]{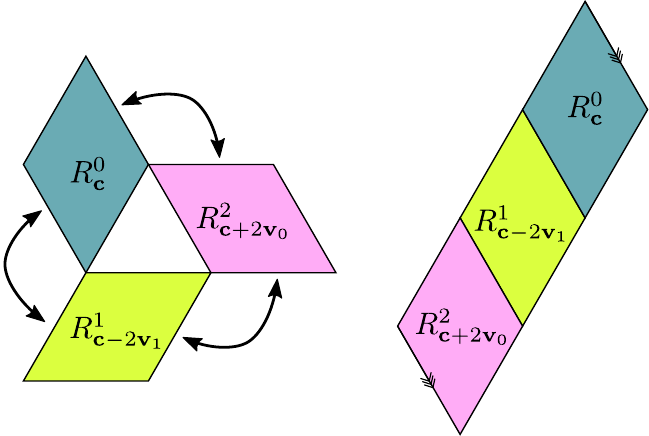}
\end{center}
\caption{Far left: A triple of rhombi surrounding an downward-pointing triangle. Arrows denote gluings used to form a cylinder in $S$. Second to left: The cylinder obtained by gluing these edges. Right two figures: The corresponding pictures for rhombi surrounding an upward-pointing triangle. Colors were chosen to match Figure \ref{fig:pinwheel}.}
\label{fig:cylinder}
\end{figure}

Let $Y$ be the surface formed by gluing together opposite sides of $R^0_\0$ by translation.
Since these edges are parallel, the edges are glued by translation and the surface $Y$ is a torus. The torus $Y$ is depicted in Figure \ref{fig:tori}. We use $Y^\circ$ to denote the torus $Y$ with the single point formed by identifying the vertices of the rhombus removed.

\begin{figure}
\begin{center}
\includegraphics[height=2in]{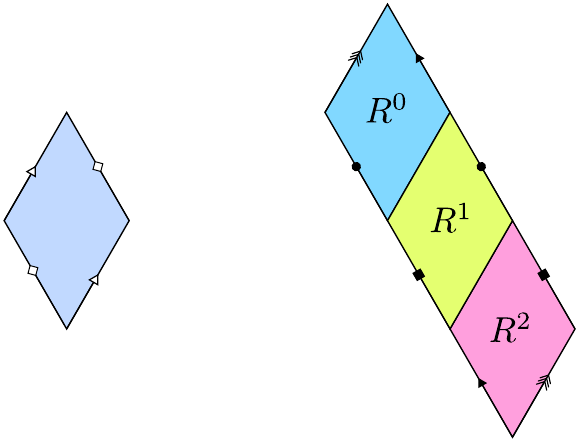}
\caption{The tori $Y$ (left) and $Z$ (right).}
\label{fig:tori}
\end{center}
\end{figure}

\begin{proposition}
\label{prop:covering}
The map from the union of interiors of rhombi to the interior of $R^0_\0$ which carries the interior of each
$R^i_\c$ isometrically to the interior of $R^0_\0$ in a manner which respects the edge labels by $\{1,2,3,4\}$
extends to a covering map $\pi:S \to Y^\circ$.
\end{proposition}
\begin{proof}
First we must check that the map from the interior of a rhombus $R^i_\c$ to $R^0_\0$ is well defined.
Recall from the definition of $R^i_\c$ and the labeling, the center of the hexagon $H_\c$ is the vertex of $R^i_\c$ which belongs to the edges with labels $1$ and $4$. This is always an obtuse angle. Since edges of rhombi are labeled by
$\{1,2,3,4\}$ in counterclockwise order, all rhombi are isometric by a orientation-preserving and label-preserving isometry.
Thus the map $\pi$ is well defined on the interiors of rhombi.

To see that we can extend to the boundaries, note that the edge gluings of both $S$ and $Y$ are by orientation-preserving Euclidean isometries (rotations for $S$ and translations for $Y$). From \eqref{eq:gluings} we observe that
the gluing rules always identify an edge labeled $1$ with an edge labeled $3$ and always identify an edge labeled $2$ with an edge labeled $4$. This was also the choice used to form $Y$ from $R^0_\0$. Therefore, we can extend the map to edges in a well defined way.
\end{proof}

A {\em translation surface} is a topological surface with an atlas of charts to the plane so that the transition functions are translations. A surface built out of polygons in the plane with vertices removed and with edges glued in pairs by translations is naturally a translation surface. When finitely many polygons
are identified, cone singularities typically appear with cone angles in $2 \pi \Z$. When infinitely many isometric polygons are used you will often see infinite cone singularities as well, which is what we will see below.

\comref{Corollary 3.3 and above, page 15: the definition of translation surface is
too rushed. In particular, it seems to me that with your definition the
singularities are not part of the surface. Finally, could you add some
details to Corollary 3.3? It seems that two small sets in S that project to
the same set in Y would lift to the same set in the plane R 2 .}

\begin{corollary}
\label{cor:translation surface}
The surface $S$ is isometric to a translation surface.
\end{corollary}
\begin{proof}
Any branched cover of a flat torus is naturally a translation surface. Our surface $S$ covers the punctured torus $Y$; see Proposition \ref{prop:covering}.
\end{proof}

Figure \ref{fig:translation_surface} depicts
$S$ but with rhombi rotated and translated. Rhombi were rotated so that the rhombi differ from $R^0_\0$ by a translation respecting the edge labeling by 
$\{1,2,3,4\}$ as described earlier. Then the rhombi were translated so that they are organized into the picture above. Because edges are glued by translation,
this Figure represents an explicit presentation of a translation surface isometric to $S$. We abuse notation by identifying $S$ with this translation surface.
This is also the translation structure on $S$ obtained by pulling back the translation structure on $Y$ under the covering map $S \to Y$.

\begin{figure}
\begin{center}
\includegraphics[height=4in]{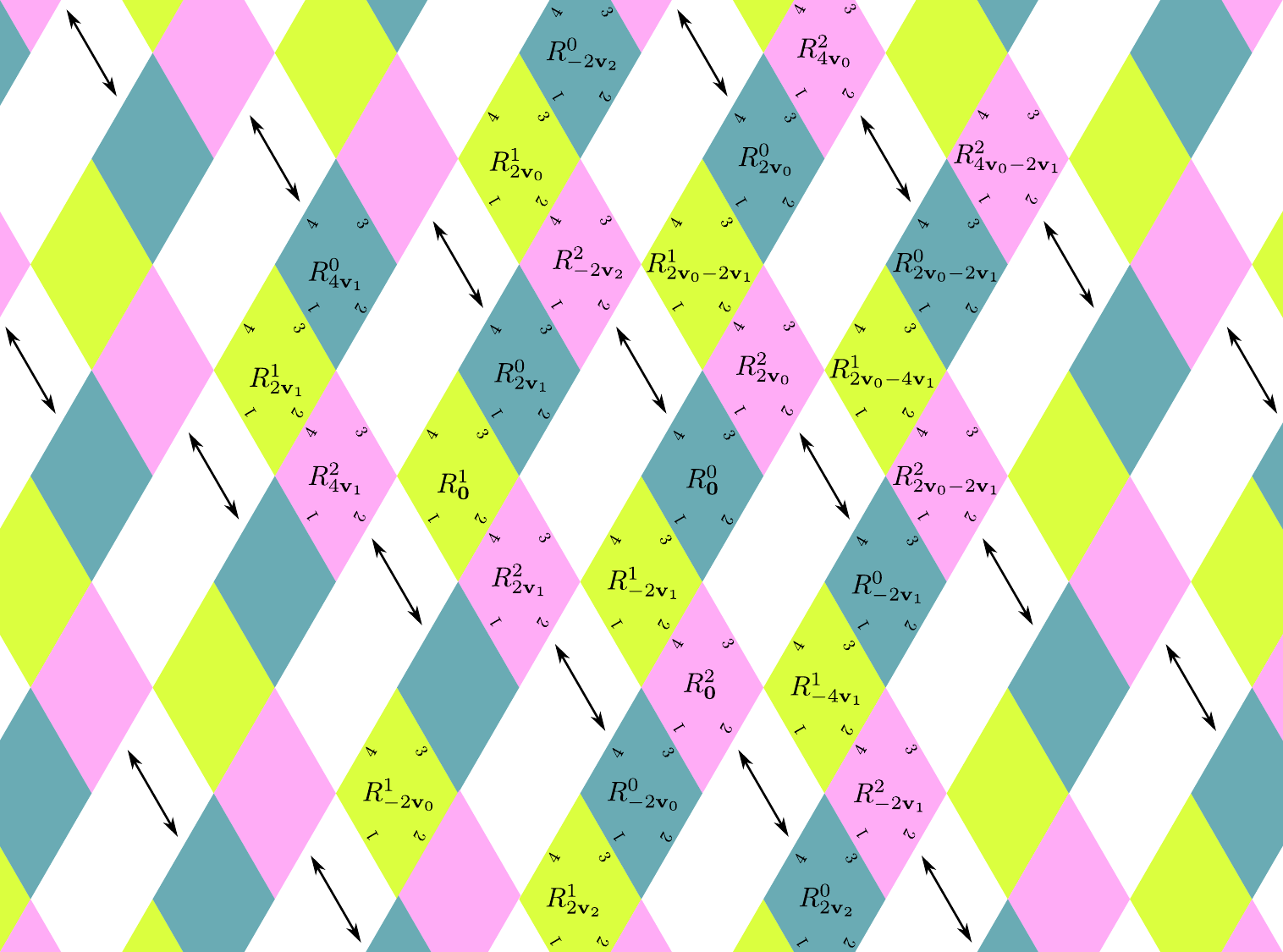}
\caption{The surface $S$ viewed as a translation surface. Here adjacent edges of rhombi are glued, as well as the edges connected by arrows. This, together with the fact that cylinders are formed from three rhombi in the directions parallel to the rhombi's sides, determines the surface.}
\label{fig:translation_surface}
\end{center}
\end{figure}

On a translation surface, the geodesic flow decomposes into natural invariant sets. For $\theta \in S^1$, the {\em straight-line flow} on $S$ is given in local coordinates by
$$F^s_\theta(\x)=\x+s(\cos \theta, \sin \theta).$$
The flow $F_\theta$ is said to be {\em completely periodic} if every non-singular trajectory of $F_\theta$ is periodic. For compact translation surfaces and covers of such surfaces, $F_\theta$ is completely periodic if and only if there is a decomposition of the surface into cylinders with geodesic core curves parallel to $\theta$.


We will now formalize the relationship between the refractive flow $T_{\theta,+}$ and the straight line flow $F_\theta$: these two flows are orbit equivalent.
A continuous {\em orbit equivalence} between two flows is a homeomorphism between the domains of the flows which carries trajectories to trajectories.

\begin{theorem}[Orbit equivalence]
\label{thm:orbit equivalence}
Fix an angle $\theta$ with $\frac{\pi}{3} \leq \theta \leq \frac{2\pi}{3}$. There is a continuous orbit equivalence
$\x:S \to X_{\theta,+}$ depending on $\theta$ from the straight-line flow $F_\theta^s:S \to S$
to the flow $T^t_{\theta,+}:X_{\theta,+} \to X_{\theta,+}$.
This orbit equivalence respects the orientations on trajectories provided by the flows. Furthermore, the orbit equivalence
carries the Lebesgue-transverse measure
on $S$ (transverse to the orbit foliation
of $F_\theta^s$) to the Lebesgue-transverse measure on $X_{\theta,+}$ (transverse to the orbit foliation of $T^t_{\theta,+}$). There is a constant $L$ so that
the restriction of $\x$ to any trajectory $\{F^s_\theta(p):s \in \R\}$ is an $L$-bilipschitz map
onto the trajectory $\{T^t_\theta \circ \x(p):t \in \R\}$ where the trajectories are endowed with metrics making these flows unit speed.
\end{theorem}
\begin{proof}
Fix $\theta$ as in the theorem and consider the straight-line flow $F^s_\theta$ on $S$ in direction $\theta$. Let $\Delta_S \subset S$ denote the union of the short diagonals of the rhombi making up $S$. These diagonals are horizontal when viewing $S$ as a translation surface; see Figure \ref{fig:translation_surface}.
Observe that $\Delta_S$ is a section for the flow straight-line flow $F^s_\theta$ on $S$. Given a point $p \in S \smallsetminus \Delta_S$, we associate a positive and a negative number:
$$s_+(p)=\min \{s>0:~F^s_\theta(p)\in \Delta_S\}
\quad \text{and} \quad
s_-(p)=\max \{s<0:~F^s_\theta(p)\in \Delta_S\}.$$

Now consider the flow $T_{\theta,+}:X_{\theta,+} \to X_{\theta,+}$. Recall from the introduction that $X_{\theta,+}$ is refractive flow-invariant subset of the unit tangent bundle of the plane; see the introduction.
We define $\Delta_T \subset X_{\theta,+}$ to be those unit tangent vectors based on the short diagonal of one of our rhombi and traveling in a direction making angle of $\theta$ with diagonal oriented outward from the center of the containing hexagon (corresponding to the horizontal direction on $S$). In light of Lemma \ref{lem:exit},
we see that $\Delta_T$ forms a section for the flow.

Now we will define the orbit equivalence
$\x:S \to X_{\theta,+}$. We define $\x(p)$ by cases.
First suppose that $p \in \Delta_S$. Then $p$ is some point on some short diagonal of a rhombus in $S$. The same rhombus is also a subset of a hexagon in the plane. We define $\x(p)$ to be the unit tangent vector based at the corresponding point in the plane with a direction which makes an angle of $\theta$ with the short diagonal oriented outward from the center of the hexagon. Observe that with this definition of $\x$, the map sends the Lebesgue transverse measure on
$\Delta_S$ to the Lebesgue transverse measure on $\Delta_T$. Now suppose that
$p \in S \smallsetminus \Delta_S$. Then the points
$$
p_-=F_\theta^{s_-(p)}(p)
\quad \text{and} \quad
p_+=F_\theta^{s_+(p)}(p)
$$
lie in $\Delta_S$. Define
$\x_-=\x(p_-)$ and $\x_+=\x(p_+)$ using the first case
defined above. Then $\x_-, \x_+ \in \Delta_T$.
Observe that the $p_+$ is the first return of $p_-$ to $\Delta_S$, and by Lemma \ref{lem:exit} $\x_+$ is the first return of $\x_-$ to $\Delta_T$. The quantity $s_+(p) - s_-(p)$ represents the amount of time it takes $p_-$ to reach $p_+$.
Similarly there is some time $t_\ast$ representing the time it takes $\x_-$ to reach $\x_+$. We define:
$$\x(p)=T^t(\x_-) \quad \text{where} \quad
t=\frac{-t_\ast  s_-(p)}{s_+(p) - s_-(p)}.$$
This completes the definition of $\x:S \to X_{\theta,+}$.
Observe that on the orbit segment from $p_-$ to $p_+$, we have rescaled time affinely by a factor of $r=\frac{t_\ast}{s_+(p)-s_-(p)}$, i.e.:
$$\x \circ F^s_\theta(p)=T^{rs}_{\theta,+}\circ \x(p) \quad
\text{for all $s \in [s_-(p),s_+(p)]$}.$$
See Figure \ref{fig:time change} for an illustrated example. We have that $s_+(p)-s_-(p) \leq t_\ast \leq 2+s_+(p)-s_-(p)$ because by Lemma \ref{lem:exit}
the portion of the trajectory through $\x(p)$ between returns to $\Delta_T$ is formed by adding a passage through an equilateral triangle and its reflection (also see Figure \ref{fig:pinwheel}). This proves that the deformation in the flow direction is bilipschitz and that the constant can be taken to be independent of $p$.
\end{proof}

\begin{figure}
\begin{center}
\includegraphics[width=\textwidth]{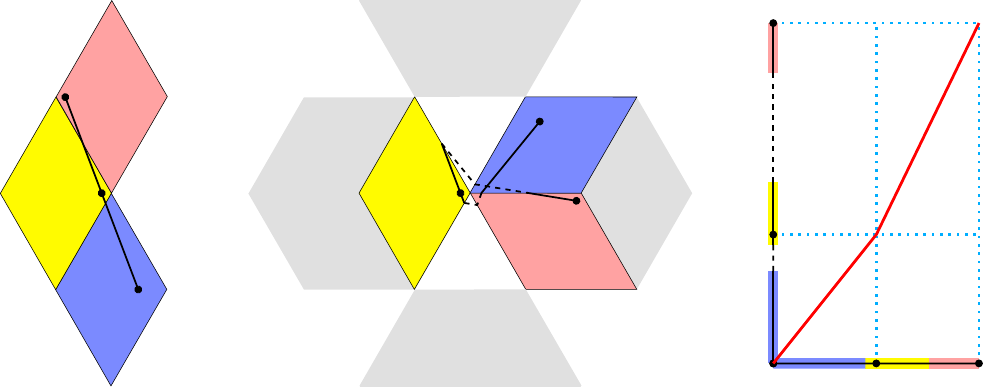}
\caption{Left: A trajectory on $S$. Center: The corresponding trajectory in the tiling. Right: The graph of the time change.}
\label{fig:time change}
\end{center}
\end{figure}

\begin{proof}[Proof of Theorem \ref{thm:bounded implies periodic} (a)]
Consider a refractive trajectory and suppose it is bounded and non-singular. Up to a symmetry of the tiling, by Proposition \ref{prop: standardization} we may assume that we are considering an orbit of $T_{\theta,+}$
for some $\theta \in [\frac{\pi}{3}, \frac{2 \pi}{3}]$. The preimage of our trajectory under the orbit equivalence of Theorem \ref{thm:orbit equivalence} is a bounded non-singular straight-line trajectory on $S$.

Since the trajectory on $S$ is bounded, it only visits finitely many rhombi making up the surface. We can then build a closed (compact) translation surface $S'$ using only the rhombi in $S$ which our trajectory intersects. We need to specify the edge gluings for $S'$. If both of two adjacent rhombi in $S$ are included in $S'$ then we glue them together in the same way. This will leave some edges unglued and since the trajectory does not cross these edges, we can glue them together in an arbitrary way ensuring that we get a translation surface. With this definition of $S'$ our trajectory also represents a trajectory on $S'$.

We now use the following basic fact: The closure of a non-periodic and non-singular straight-line trajectory on a flat surface is either the full translation surface or a subsurface bounded by saddle connections in the direction of the trajectory \cite{FK} \cite[Proof of Theorem 1.8]{masur tabachnikov}.
Our trajectory in $S'$ is not dense in $S'$, so it is either periodic (Case I), or dense in a subsurface of $S'$ bounded by saddle connections parallel to the trajectory (Case II).

Case I is our desired conclusion, so assume by contradiction that Case II holds, i.e., our trajectory is dense in a subsurface of $S'$ bounded by parallel saddle connections.
The surface $S'$ is a finite cover of the torus built from one of our rhombi. Thus, all saddle connections are parallel to a vector in $\Lambda$, and in addition, straight-line flows on $S'$ in directions in $\Lambda$ are completely periodic since $S'$ is a finite branched cover of $\R^2/\Lambda$.
It follows that our trajectory on $S'$ is parallel to a vector in $\Lambda$ and is therefore periodic. This contradicts the density of our trajectory in a subsurface and rules out Case II.
\end{proof}

\section{Hidden symmetries}\label{hidden symmetries}

\subsection{Background on symmetries of translation surfaces}
A {\em translation automorphism} of a translation surface is a homeomorphism from the surface to itself which preserves the translation structure. A homeomorphism is a translation automorphism if and only if it acts as a translation in local coordinate charts. The collection of all translation automorphisms of a surface form a group $\Trans(S)$.

An {\em affine automorphism} of a translation surface $S$ is a homeomorphism from the surface to itself which preserves the affine structure underlying the translation surface structure. In other words, the homeomorphism must act affinely in local coordinates. In a connected translation surface, this means that there is a matrix $M \in \GL(2, \R)$ such that in local coordinate charts the homeomorphism $h$ has the form
$$h\left(\begin{array}{r} x \\ y\end{array}\right)=M\left(\begin{array}{r} x \\ y\end{array}\right)+\left(\begin{array}{r} c_0 \\ c_1\end{array}\right),$$
where $(c_0,c_1)$ is a vector that depends on the charts. The matrix $M$ is independent of the chart, and we call this the {\em derivative} of the affine automorphism, $D(h)$.

The group of affine automorphisms of a translation surface forms a group $\Aff(S)$, and the group $V(S)=D\big(\Aff(S))\subset \GL(2,\R)$ is the {\em Veech group} of $S$. We observe that $\Trans(S)$ is a normal subgroup
of $\Aff(S)$ since it is the kernel of the derivative homomorphism $D:\Aff(S) \to \GL(2,\R)$, and $V(S)$ is isomorphic
to the quotient $\Aff(S)/\Trans(S)$.

A cylinder $C \subset S$ is a subset of a translation surface isometric to $\R/c\Z \times (0,h)$. We call the constant $c$ the {\em circumference} of the cylinder and $h$ the {\em height}. The {\em inverse modulus} of $C$ is the ratio $c/h$. Core curves in $\R/c\Z \times (0,h)$ have the form $\R/c\Z \times \{y\}$ for some $y \in (0,h)$. These are closed straight-line trajectories on the surface and we say the {\em direction} of the cylinder is the direction of travel of these trajectories in the projectivization $P \R^2$. A {\em decomposition} of $S$ into cylinders is a collection of disjoint cylinders $\{C_i\}$ with the same direction so that the collection of closures of cylinders covers $S$.

\begin{proposition}[Thurston \cite{thurston}]
\label{prop:dehn twist}
Suppose that a translation surface $S$ admits a decomposition into cylinders in the direction of the unit vector $\u$ where all cylinders have the same inverse modulus, $\lambda$. Then there is an affine automorphism $\phi$ of $S$ which performs a single right Dehn twist in all cylinders in the decomposition and
$$D(\phi)=R \circ \twotwo 1\lambda01 \circ R^{-1}$$
where $R$ is the rotation carrying the vector $(1,0)$ to $\u$.
\end{proposition}

\subsection{An abelian covering}
The subgroup $2 \Lambda \subset \R^2$ acts on the plane by translation and preserves the tiling. It also preserves the collection of rhombi. Indeed, we get an action of $2\Lambda$ on rhombi defined so that for $\w \in 2\Lambda$,
\begin{equation}
\label{eq:tau}
\tau_\w(R^i_\c)=R^i_{\c+\w} \quad \text{for all $i \in \{0,1,2\}$ and all $\c \in 2 \Lambda$.}
\end{equation}
Furthermore $\tau$ preserves edge labels. It follows:

\begin{proposition}
The action $\tau$ of $2 \Lambda$ on rhombi induces an action of $2\Lambda$
on $S$ by translation automorphisms.
\end{proposition}
\begin{proof}
Observe that translation by $\w \in 2 \Lambda$ preserves edge gluings; see (\ref{eq:gluings}). Since labels of edges are respected, each automorphism acts as a translation in local coordinates of the translation surface.
\end{proof}

Define $Z^\circ=S/2 \Lambda$. Since the $\tau$-orbit of each $R^i_\c$ consists of all $R^i_\w$ with $\w \in 2\Lambda$, the surface $Z^\circ$ consists of three rhombi with edges identified and vertices removed. Topologically $Z^\circ$ is a torus punctured at three points (points appearing as vertices of rhombi). We define $Z$ to be the torus obtained by adding the three points to the surface. See Figure \ref{fig:tori}.

By construction, $\pi:S \to Z^\circ$ is a regular cover with deck group $2 \Lambda$. Since the cover is regular, we can define the {\em monodromy homomorphism} from the fundamental group. We choose a basepoint on $z_0 \in Z^\circ$ and a basepoint on $s_0 \in S$ so that $\pi(s_0)=z_0$. Given a loop $\gamma$ based at $z_0$, we can lift to a curve $\tilde \gamma$ starting at $s_0$ and terminating at a point $h(\gamma)(z_0)$ where $h(\gamma)$ is an element of the deck group. Because our deck group $2\Lambda$ is abelian, the monodromy homomorphism is well defined as a map on homology,
\begin{equation}
\label{eq:h}
h:H_1(Z^\circ; \Z) \to 2 \Lambda.
\end{equation}

We will now provide an explicit description of the map $h$. Let $\Sigma \subset Z$ be the three points of $Z \smallsetminus Z^\circ$. Recall that algebraic intersection number gives a
non-degenerate pairing
$$\cap:H_1(Z,\Sigma; \Z) \times H_1(Z^\circ; \Z) \to \Z.$$
\begin{proposition}
\label{prop:monodromy}
Let $\eta_0, \eta_1 \in H_1(Z,\Sigma; \Z)$ be the two relative homology classes depicted on the right side of Figure \ref{fig:torus fundamental group}. We have
$$h(\gamma)=2(\eta_0 \cap \gamma)\v_0+2(\eta_1 \cap \gamma)\v_1,$$
where $\v_0$ and $\v_1$ are as in \eqref{eq:e}.
\end{proposition}

\begin{figure}
\begin{center}
\includegraphics[width=\textwidth]{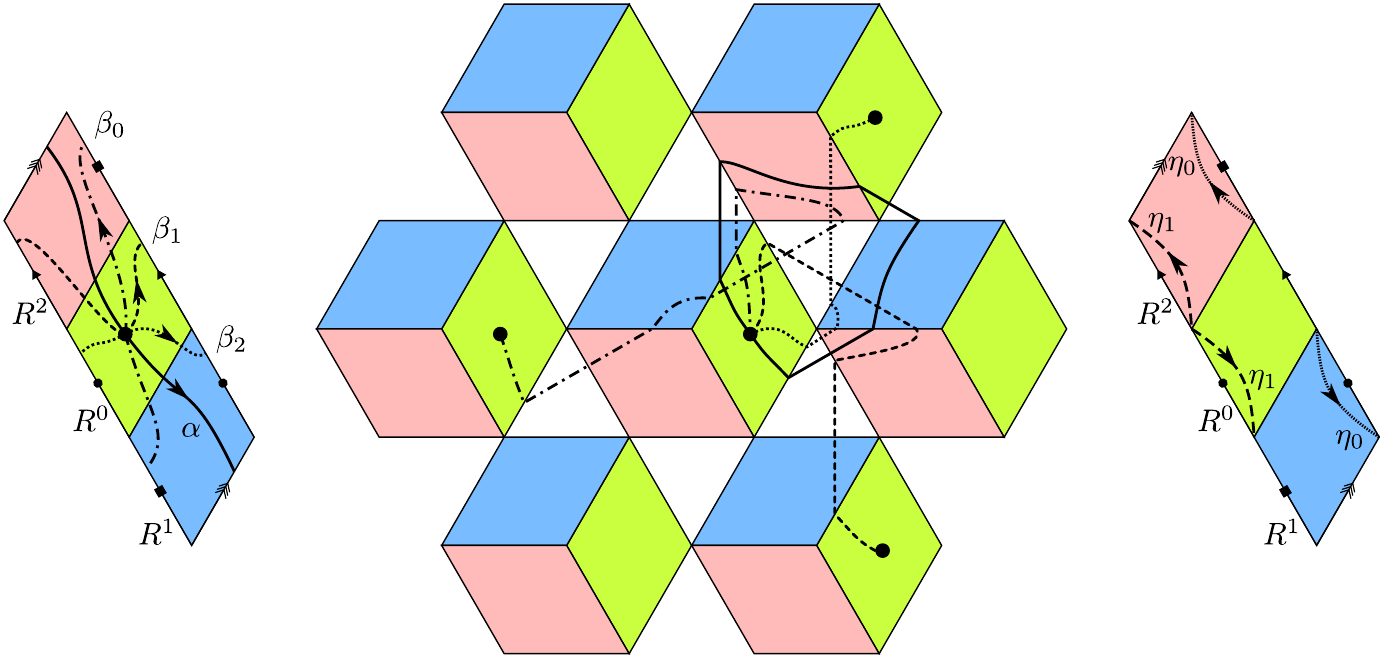}
\caption{Left: Curves generating the fundamental group of $Z^\circ$. Center: These same curves on the tiling with line segments joining the arcs within rhombuses according to the gluings defining $S$. Right: The relative homology classes $\eta_0$ and $\eta_1$.}
\label{fig:torus fundamental group}
\end{center}
\end{figure}

\begin{proof}
By linearity, it suffices to check the equation for
the basis $\{\alpha, \beta_0, \beta_1, \beta_2\}$
of $H_1(Z^\circ; \Z)$ which consists of the closed curves
in $Z^\circ$ depicted in the left side of Figure \ref{fig:torus fundamental group}. First we need to
see what $h$ does to this basis. For this, we need to lift the curves to $S$ and find a deck group element for each curve as noted above. Since the deck group was defined using the tiling, we lift the curves to $S$, and then carry the curves to the tiling using the rhombi. By definition of $S$, this is equivalent to developing the curves into the tiling and whenever you cross an edge you use the edge gluings of $S$ to decide how to develop across the edge. This is carried out in the center of Figure \ref{fig:torus fundamental group}. We find that
\begin{equation}
\label{eq:h on basis}
h(\alpha)=\0 \quad \text{and} \quad h(\beta_i)=-2 \v_i \quad \text{for $i \in \{0,1,2\}$}.
\end{equation}
This evaluates the left side of the equation in the proposition on the basis.

The right side of the equation in the proposition involves
intersection numbers with the classes $\eta_0$ and $\eta_1$. Observe that $\alpha$ does not intersect these classes, while
$$\eta_0 \cap \beta_0=\eta_1 \cap \beta_1=-1,
\quad
\eta_0 \cap \beta_1=\eta_1 \cap \beta_0=0,
\quad \text{and} \quad
\eta_0 \cap \beta_2=\eta_1 \cap \beta_2=1.$$
Using these algebraic intersection numbers to evaluate the expression on the right side of the equation yields the same results as \eqref{eq:h on basis}.
\end{proof}

\begin{proof}[Proof of Theorem \ref{thm:bounded implies periodic} (b)]
\comref{The proof of Theorem 1.8 (b) at page 19 seems to assume the, not yet
proven, Theorem 1.6. If this is the case this should be stated clearly. In
the proof of Theorem 1.8 (b) you say that “Rhombi with the same center
c form a cylinder.” This is not true for all directions, for instance not for
the ones generating the lattice, and seems not so apparent in general,
could you clarify this?}
Consider a non-singular trajectory of $T$ which we may take to lie in $X_{\theta,+}$ for some $\theta$ with $\frac{\pi}{3} \leq \theta \leq \frac{2\pi}{3}$. It is an elementary exercise to show that the linear drift rate is zero for a periodic trajectory and positive for a drift-periodic trajectory.
So we can assume $\theta$ is not parallel to a vector in $\Lambda$. (This is a consequence of the fact that the flow is semi-conjugate to straight-line flow on $S$ and $S$ is a periodic cover of the torus $Y$ of Figure \ref{fig:tori}. It is an elementary observation that a unit vector in $\R^2$ is a parallel to a vector in $\Lambda$ if and only if the straight-line flow in that direction is periodic on $Y$. Thus all trajectories on $S$ in these directions are periodic or drift periodic. \compat{This is in response to referee's concerns on this depending on Theorem 1.6.})

By Theorem \ref{thm:orbit equivalence}, our trajectory is the image under a bi-Lipschitz orbit equivalence $\x:S \to X_{\theta,+}$ of a straight-line trajectory in $S$.
\compat{I had erroneously assumed that $R^0_\c \cup R^1_\c \cup R^2_\c$ formed a cylinder. Corrected:}
Any $\c \in 2 \Lambda$ determines
a cylinder on $S$ in direction $v_1$ namely $C_\c=R^0_\c \cup R^1_{\c+2\v_0} \cup R^2_{\c-2\v_2}$ on $S$; see the left side of Figure \ref{fig:cylinder}. The collection of such cylinders $\{C_\c:~ \c \in 2\Lambda\}$ forms a cylinder decomposition in direction $\v_1$. Such cylinders are all isometric and parallel and so are crossed in constant time (depending on $\theta$).
Observe that there is a uniform upper bound on the distance from $\x(s)$ to $\c$ where $\c \in 2\Lambda$ is defined so that $\x(s) \in C_\c$. (This holds by definition of the orbit equivalence.)
Let $C_{\c_n}$ be the sequence of cylinders crossed by the $T$-orbit of $\x(s_0)$ for some $s_0 \in S$. Taken all together, we see that
\begin{equation}
\label{eq:equivalence}
\lim_{t \to \infty} \frac{\big|T^t\big(\x(s_0)\big)\big|}{t}=0
\quad \text{if and only if} \quad
\lim_{n \to \infty} \frac{|\c_n|}{n}=0.
\end{equation}

Consider the projection of the $F_\theta$-trajectory of $s_0$ to the punctured torus $Z^\circ$. The sequence of centers $\c_n$ can be recovered by intersecting increasing segments of this trajectory with the curves $\eta_0$ and $\eta_1$ of Figure \ref{fig:torus fundamental group}. This is the content of Proposition \ref{prop:monodromy}.
Note that the curves can be moved onto the boundary of the cylinder (the negatively sloped boundary edges of the rhombi). Consider the return map of flow in direction $\theta$ to the union of negative sloped boundaries of rhombi making up $Z$. From assumptions in the first paragraph, this is an irrational rotation.
Then $\c_n$ is determined by Birkhoff averages of two functions which take the values $1$, $-1$ and zero on intervals each making up one third of the circle. (These functions come from the direction in which the curves $\eta_0$ and $\eta_1$ move over the boundary components.)
Since the functions have zero integral, their time average value is zero in the sense of the Birkhoff ergodic theorem. This verifies that the right side of
\eqref{eq:equivalence} is true. The left side is the desired conclusion.
\end{proof}

\subsection{Lifting affine automorphisms}
Since $Z^\circ$ is ``parallelogram-tiled,''
it has a lattice Veech group. This guarantees there are many affine automorphisms of $Z^\circ$. We now consider which of these lift to $S$. It turns out they all lift:

\begin{lemma}
\label{lem:lifting}
Let $\pi:S \to Z^\circ$ denote the covering map.
For any affine automorphism $f:Z^\circ \to Z^\circ$,
there is an affine automorphism $\tilde f:S \to S$
so that $f \circ \pi=\pi \circ \tilde f$.
\end{lemma}

We carry out the proof using ideas from \cite{hooper weiss}. It follows from work there that there can be only one $\Z^2$-cover of $Z^\circ$ so that straight-line flow recurs in almost every direction; see \cite[\S 4]{hooper weiss}. The surface $S$ turns out to be this unique cover.

\begin{proposition}
\label{prop:six special classes}
The collection of six homology classes
$$\{\eta_0,-\eta_0, \eta_1, -\eta_1, \eta_0-\eta_1, -\eta_0 + \eta_1\} \subset H_1(S,\Sigma;\Z)$$
is invariant under the action of the affine automorphism group of $Z^\circ$.
\end{proposition}

For the proof we need the concept of holonomy.
The {\em holonomy} of a curve in a translation surface obtained by developing the curve into the plane using local coordinates and then measuring the vector difference between the start and end points. This concept of holonomy induces a linear mapping $\hol:H_1(Z,\Sigma; \Z) \to \R^2$.

\begin{proof}
We will explain why these six classes are canonical
in an affine-invariant sense. First of all they have trivial holonomy which is certainly an affinely invariant concept. Second, consider the boundary map
$$\delta: H_1(Z,\Sigma;\Z) \to H_0(\Sigma; \Z).$$
Note that $\delta$ is equivariant under the induced actions of a homeomorphism of $(Z,\Sigma)$.
Furthermore, the action of homeomorphisms on $H_0(\Sigma; \Z)$ is by permutation matrices.
We observe that a class in $H_1(Z,\Sigma;\Z)$
is determined by its holonomy and its image under $\delta$. This is because two elements with the same
image under $\delta$ differ by absolute homology classes
in $H_1(Z;\Z)$, and non-trivial absolute homology classes have non-trivial holonomy since $Z$ is a flat torus.

\comref{Proposition 4.5, page 20: in the last paragraph you consider elements of
H 0 (Sigma; Z) as functions in -1, 0, 1. Are you passing to cohomology?} \compat{I should have been referring to coefficients in a sum. Corrected.}
Let $p_0$, $p_1$ and $p_2$ denote the three points of $\Sigma$ and let $[p_i]$ denote the corresponding classes in $H_0(\Sigma;\Z)$.
Observe that the images under $\delta$ of the six classes in the Proposition have the form
\begin{equation}
\label{eq:points}
c_0[p_0]+c_1[p_1]+c_2[p_2]\quad \text{where} \quad \{c_0,c_1,c_2\}=\{-1,0,1\},
\end{equation}
i.e., each coefficient is in the set $\{-1,0,1\}$ and each coefficient appears once.
Observe there are six elements of $H_0(\Sigma;\Z)$ of this form and these six classes coincide with the six stated in this proposition.
This collection is clearly invariant under the permutation action on $\Sigma$.
Thus, the six relative homology classes listed in the proposition are invariant under the action of affine automorphisms of $Z^\circ$: They are precisely those classes with trivial holonomy and whose images under $\delta$ have expressions as in \eqref{eq:points}.
\end{proof}

\begin{proof}[Proof of Lemma \ref{lem:lifting}]
Let $f:Z^\circ \to Z^\circ$ be a homeomorphism. From covering theory, $f$ lifts to a homeomorphism $\tilde f:S \to S$ if and only if the induced action on homology
$f_\ast:H_1(Z^\circ; \Z) \to H_1(Z^\circ; \Z)$ preserves the kernel $\ker h$, where $h$ is as defined in \eqref{eq:h}. By Proposition \ref{prop:monodromy},
$\ker h$ consists of those classes whose algebraic intersection numbers with $\eta_0$ and $\eta_1$ are zero. Observe that the span of $\eta_0$ and $\eta_1$
in $H_1(Z,\Sigma;\Z)$ is $2$-dimensional and coincides with the span of the six classes of Proposition \ref{prop:six special classes}. When $f$ is an affine homeomorphism, these six classes are preserved by the action of $f$ and it follows that their span $W \subset H_1(Z, \Sigma; \R)$ is also preserved.
The kernel can then be written $$\ker h=\{\gamma \in H_1(Z^\circ; \Z):~\eta \cap \gamma=0 \quad \text{for all $\eta \in W$}\},$$
which is $f_\ast$-invariant because $W$ is invariant.
\end{proof}

\subsection{Hyperbolic geometry}
\label{sect:hyperbolic geometry}

To understand and visualize the Veech groups of translation surfaces, it is useful to consider the hyperbolic geometry of these groups.

We will be using the fact that $\PGL(2,\R)$ is the isometry group of the hyperbolic plane $\H^2=O(2) \backslash \PGL(2,\R)$. From this point of view, the isometric action is given by right multiplication. The subgroup $\PSL(2,\R) \subset \PGL(2,\R)$ forms the unit tangent bundle of the hyperbolic plane. The geodesic flow on the unit tangent bundle is then given by the left multiplication action of
\begin{equation}
\label{eq:gt}
g_t=\left[\begin{array}{rr}
e^{-t} & 0 \\
0 & e^t\end{array} \right]
\end{equation}
on $\PSL(2,\R)$.

The upper half-plane model of $\H^2$ identifies points in $\H^2$ with complex numbers with positive imaginary part. To be concrete, we choose our identification so that $M \in \PSL(2,\R)$ represents a unit tangent vector based at
\begin{equation}
\label{eq:upper half-plane}
\frac{di-b}{a-ci} \quad \text{when} \quad M=\twotwo{a}{b}{c}{d};
\end{equation}
this point coincides with the projectivization of the vector $M^{-1}(i,1) \in \C^2$. In particular, the isometric action of $\PGL(2,\R)$ on the upper half plane is given by
$$\twotwo{a}{b}{c}{d}:~z \mapsto \begin{cases}
\frac{dz-b}{a-cz} & \text{if $ad-bc>0$},\\
\frac{d\bar{z}-b}{a-c\bar{z}} & \text{if $ad-bc<0$}.
\end{cases}
$$
(The action of $M \in \PSL(2,\R)$ is by the inverse of the usual M\"obius transformation associated to $M$.)

With these conventions, $i$ is the point in the upper half-plane represented by the coset $\O(2)$ of $\H^2=O(2) \backslash \PGL(2,\R)$.
Let $r_\theta \in O(2)$ be the vector which rotates
the vector $(1,0)$ to $(\cos \theta, \sin \theta)$ .
As $t \to \infty$, the $\theta$-direction of the plane is contracted by $g_t r_\theta^{-1}$. This is the geodesic leaving $i$ in the upper half-lane at an angle of $-2 \theta$ from the vertical. The geodesic limits on $\cot(\theta) \in \R \cup \{\infty\}$ as $t \to +\infty$. (This geodesic coincides with the $g_t(\vec{u_\theta})$ when $0 \leq \theta \leq \frac{\pi}{2}$ in the introduction.)

\subsection{The affine automorphism groups}
\comref{Finally, when in Section 4.5 you use
right Dehn twists along cylinders in directions v i you could reference the
pictures you included of the cylinders to help make this distinction clear.}
\compat{I think the referee is referring to our image of cylinders on $S$. I did add a reference to the figure showing $Z$.}
Recall that $Z^\circ$ is the torus built out of three rhombi with the vertices removed; see Figure \ref{fig:tori}.
We will now work out some facts about the affine automorphism group and the Veech group of $Z^\circ$.

We observe using Proposition \ref{prop:dehn twist} that $\Aff(Z^\circ)$ contains the following elements:
\begin{itemize}
\item The right Dehn twist $\phi_1$ in the single maximal cylinder in direction $\e_1$.
\item The right Dehn twist $\phi_2$ in the single maximal cylinder in direction $\e_2$.
\item The right Dehn twist $\phi_0$ in the three maximal cylinders in direction $\e_0$.
\end{itemize}
In addition, by inspection we can find the following orientation-reversing element:
\begin{itemize}
\item There is an affine automorphism $\rho$ of $Z^\circ$ which preserves each of the three rhombi making up $Z$ and acts as a reflection in the short diagonal of each rhombus.
\end{itemize}
We let $P_i=D(\phi_i)$ and $R=D(\rho)$ and by computation find that in the basis $\{\e_1, \e_2\}$:
\begin{equation}
\label{eq:some matrices}
P_1=\twotwo 1301,
\qquad
P_2=\twotwo 10{-3}1,
\qquad
P_0=\twotwo 01{-1}2,
\qquad
R=\twotwo 0{1}{1}0.
\end{equation}
Finally, we note that there is an affine automorphism with derivative $-I$ which preserves rhombus $R_0$ and swaps rhombus $R_1$ with $R_2$. This is convenient because it means that $V(Z^\circ) \subset \GL(2,\R)$ is the preimage of a subgroup $PV(Z^\circ) \subset \PGL(2,\R)$ under the natural map $\GL(2,\R) \to \PGL(2,\R)$.

\begin{proposition}
\label{prop:generators}
The projectivized Veech groups $PV(Z^\circ)$ and
$PV(S)$ are isomorphic to the reflection group in a $(3,\infty,\infty)$-triangle (i.e., a hyperbolic triangle with one angle of $\frac{\pi}{3}$ and two ideal vertices). Reflections generating the Veech group are given by $R$, $R P_0$, and $P_1^{-1} R P_0$.
\end{proposition}

\begin{figure}
\begin{center}
\includegraphics[height=2.25in]{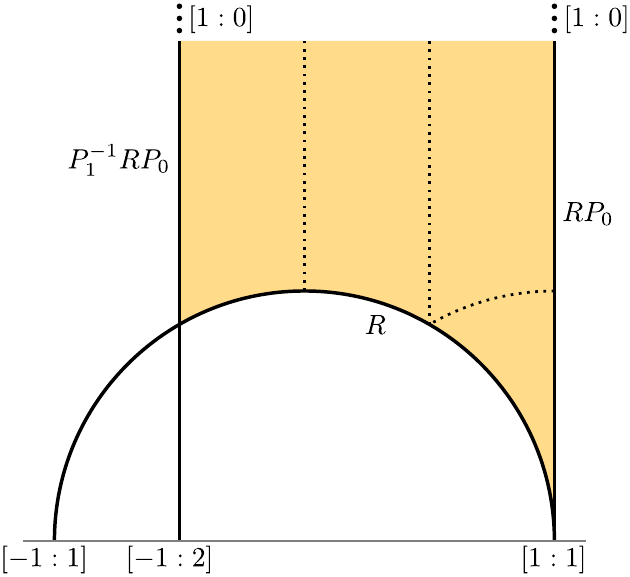}
\caption{A fundamental domain for the action of $PV(Z^\circ)$ on the hyperbolic plane. Edges are labeled by matrices reflecting in that edge. The domain is divided into fundamental regions for the action of $PGL(2,\Z)$ by isometry.}
\label{fig:veech group}
\end{center}
\end{figure}

\begin{proof}[Proof]
By Lemma \ref{lem:lifting}, it suffices to prove the statement for $PV(Z^\circ)$. A matrix calculation reveals
$$R P_0=\twotwo {-1}201 \quad \text{and} \quad P_1^{-1} R P_0=\twotwo {-1}{-1}01.$$
We consider the standard matrix action on the hyperbolic plane realized as the upper half plane bounded by the projectivization of $\R^2$. It may be verified that a fundamental domain for this action is shown in Figure \ref{fig:veech group}.

From remarks above the proposition we know that $G=\langle R, R P_0, P_1^{-1} R P_0\rangle$ is contained in $PV(Z^\circ)$. To see that $G=PV(Z^\circ)$ observe that because $Z^\circ$ covers the torus (which is ``square'' with respect to the basis $\{\e_1, \e_2\}$), we must have $PV(Z^\circ) \subset \PGL(2,\Z)$. By a covolume calculation one can see that $\langle R, R P_0, P_1^{-1} R P_0\rangle$ is index four inside $\PGL(2,\Z)$.
But $PV(Z^\circ)$ must be at least index three inside $\PGL(2,\Z)$ since there are elements of $M \in \PGL(2,\Z)$
so that $M, M^2 \not \in PV(Z^\circ)$ but $M^3 \in PV(Z^\circ)$. (If the left cosets $M \cdot PV(Z^\circ)$
and $M^2 \cdot PV(Z^\circ)$ are not distinct then you can show $M \in PV(Z^\circ)$.)
A primary example of such an $M$ is given by
\begin{equation}
\label{eq:M}
M=\twotwo 0{-1}1{-1}
\end{equation}
which cyclically permutes the vectors $\e_0$, $\e_1$ and $\e_2$. Note that $\e_0$ is geometrically different that $\e_1$ and $\e_2$ in that the cylinder decompositions are different; see the definitions of $\phi_i$ above.
By multiplicativity of subgroup indices,
$$\big[G:\PGL(2,\Z)\big]=\big[G:PV(Z^\circ)\big] \cdot \big[PV(Z^\circ):\PGL(2,\Z)\big].$$ Since $\big[PV(Z^\circ):\PGL(2,\Z)\big]\geq 3$ and $\big[G:\PGL(2,\Z)\big]=4$,
we must have $\big[PV(Z^\circ):\PGL(2,\Z)\big]=4$ and $\big[G:PV(Z^\circ)\big]=1$, i.e., $G=PV(Z^\circ)$.
\end{proof}

\begin{corollary}
\label{cor:Delta}
The region $\Delta \subset \H^2$ of Figure \ref{fig:dark-triangles} is a fundamental domain for the action of $PV(Z^\circ)$ (written in the standard basis) on $\H^2$.
\end{corollary}
\begin{proof}
Above we have done calculations in the basis $\{\v_1, \v_2\}$. Matrices in the standard basis can be obtained from matrices in basis $\{\v_1, \v_2\}$ by conjugating by the matrix $C$ whose columns are $\v_1$ and $\v_2$. In particular multiplication by $C$ carries a fundamental domain for $PV(Z^\circ)$ in the basis $\{\v_1, \v_2\}$
to the fundamental domain in the standard basis. Let $R$
be the matrix from (\ref{eq:some matrices})
which is an element of $V(Z^\circ)$ written in the basis $\{\v_1, \v_2\}$. To obtain $\Delta$, we apply $C R$ to the fundamental domain in the basis $\{\v_1, \v_2\}$ shown in Figure \ref{fig:veech group}.
\end{proof}

The proposition above found generators for $PV(Z^\circ)$. We will now give a method for distinguishing when an element $M \in \GL(2,\Z)$ represents an action of
an element in $V(Z^\circ)$ in the basis $\{\v_1, \v_2\}$. Let
$\Z^2_\vis$ denote the set of non-zero pairs $(m,n) \in \Z^2$ consisting of points visible from the origin in $\R^2$. Equivalently, $\Z^2_\vis$ is those pairs $(m,n)$
which are not both zero and satisfy $gcd(m,n)=1$. Observe that
$\Z^2_\vis$ is $\GL(2,\Z)$-invariant. We define
\begin{equation}
\label{eq:Xi}
\Xi=\{(m,n) \in \Z^2_\vis~:~m \equiv n \pmod{3}\}.
\end{equation}

\begin{theorem}
\label{thm:veech group}
The Veech groups $V(Z^\circ)$ and $V(S)$, thought of as a subset of $\GL(2,\Z)$ by writing the elements in the basis $\{\e_1, \e_2\}$, are given by $\{M \in \GL(2,\Z) ~:~ M(\Xi)=\Xi\}.$
Furthermore, the actions of these groups on $\Xi$ and $\Z^2_\vis \smallsetminus \Xi$ are both transitive.
\end{theorem}
\begin{proof}
Write $G=\{M \in \GL(2,\Z) ~:~ M(\Xi)=\Xi\}.$
By Lemma \ref{lem:lifting}, it suffices to prove the statement for $V(Z^\circ)$. To see $V(Z^\circ) \subset G$, it suffices to show that
the generators of $V(Z^\circ)$ lie in $G$. By Proposition \ref{prop:generators} the group elements $R$, $P_0$, $P_1$, and $P_2$ of \eqref{eq:some matrices} together with $-I$ generate $V(Z^\circ)$. Each can be shown to lie in $G$ by a simple calculation which we demonstrate for the case of $P_0$. Suppose $(m,n) \in \Xi$. Then $\gcd(m,n)=1$ and there is an integer $k$ so $n=m+3k$. Observe $P_0(m,n)=(n,2n-m)$. Since $P_0 \in \GL(2,\Z)$ we know that $\gcd(n,2n-m)=1$. Observe $2n-m=n+(n-m)=n+3k$ so $n \equiv 2n-m \pmod{3}$. This proves
$P_0 \in G$.

We have shown that $V(Z^\circ) \subset G$.
To show $V(Z^\circ) = G$ we use an index argument similar to the end of the proof of Proposition \ref{prop:generators}. We will show $[V(Z^\circ):G]=1$. From the paragraph above
we have $V(Z^\circ) \subset G \subset \GL(2,\Z)$ and therefore
\begin{equation}
\label{eq:index2}
[V(Z^\circ):\GL(2,\Z)]=[V(Z^\circ):G] \cdot [G:\GL(2,\Z)].
\end{equation}
From our explicit description of $V(Z^\circ)$ above we know that $[V(Z^\circ):\GL(2,\Z)]=4$. Since subgroup indices are positive integers, it suffices to prove that $[G:\GL(2,\Z)]>2$.

Consider the matrix $M \in \GL(2,\Z)$ of \eqref{eq:M} and observe that
$(1,1) \in \Xi$ while $M(1,1)=(-1,0)$ and $M^2(1,1)=(0,-1)$. Thus $M,M^2 \not \in G$ while $M^3 \in G$. We conclude that $[G:\GL(2,\Z)] \geq 3$.
Since the index must divide four we have $[G:\GL(2,\Z)] = 4$ and so $[V(Z^\circ):G]=1$ as desired.

\comref{It seems that the paragraph beginning with ``Now we will show that G'' should go above Equation 14.}
\compat{The paragraph beginning ``Now we will show that G'' should not have been there. I removed it and made some changes to clarify the proof.}

Algebraically, cusps of the quotient of the hyperbolic plane, $\H^2/V(Z^\circ)$, are $V(Z^\circ)$-orbits in $\Z^2_\vis/{-1}$ where $-1$ is acting by scalar multiplication. (Concretely, monodromy around the cusp gives a conjugacy class of parabolics, whose collective eigenspaces constitute such a $V(Z^\circ)$-orbit in the real projective line. Since $-I \in V(Z^\circ)$, $V(Z^\circ)$ acts transitively on eigendirections rather than just eigenspaces.) By Proposition \ref{prop:generators},
there are two such cusps (also see Figure \ref{fig:veech group}). Since $-I \in V(Z^\circ)$, this also means there are two orbits of $V(Z^\circ)$ in $\Z^2_\vis$. The paragraph above shows that $\Xi$ is $V(Z^\circ)$-invariant. It follows that the two $V(Z^\circ)$-orbits corresponding to the cusps must be $\Xi$ and $\Z^2_\vis \smallsetminus \Xi$. This proves the last sentence of the theorem.
\end{proof}

\section{Periodic and drift-periodic directions}\label{periodic directions}

In this section, we explicitly describe the set $\mathcal P$ of periodic directions, and the set $\mathcal D$ of drift-periodic directions, and then give results about periodic and drift-periodic trajectories.

\subsection{Characterization of periodic and drift-periodic trajectories}

We establish the following corollary to Theorem \ref{thm:veech group}.

\begin{corollary}[Periodic directions on $S$]
\label{cor:lattice directions on S}
Let $\w=m \e_1+n \e_2 \in \Lambda_\vis$
where $\Lambda_\vis \subset \Lambda$ was defined in \eqref{eq:Lambda vis}. Let $\theta = \frac{\w}{|\w|} \in S^1$. Then:
\begin{enumerate}
\item If $m \equiv n \pmod{3}$, then the straight-line flow $F_\theta$ on $S$ is {\em completely drift-periodic} (i.e., for every non-singular $F_\theta$-trajectory there is an infinite-order translation automorphism of $S$ which preserves the trajectory.)
\item Otherwise, the straight-line flow $F_\theta$ on $S$ is completely periodic.
\end{enumerate}
On the other hand, if $\theta$ is not parallel to any vector in $\Lambda$, then $F_\theta$ has no periodic or drift-periodic trajectories.
\end{corollary}

\begin{proof}
First we consider the special case when $\theta$ is horizontal. This direction may be represented as $\v_1+\v_2$. There are three horizontal cylinders on $Z^\circ$. All lift to strips on the surface $S$. To see this, refer to the left and central part of Figure \ref{fig:torus fundamental group}; the three horizontal cylinders in $Z^\circ$ have core curves that are homologous to $\alpha+\beta_0$, $\beta_1$ and $\beta_2$. Curves representing these classes do not lift
to cylinders on $S$ because the corresponding curves in the tiling to not close up; see the figure. Instead the cylinders lift to strips invariant under deck group elements $-2 \v_0$, $-2 \v_1$ or
$-2 \v_2$ respectively. By normality of the cover, $S$
is covered by horizontal strips; i.e., the horizontal direction is completely drift-periodic.

Now consider (1). Here, we have $(m,n) \in \Xi$, see \eqref{eq:Xi}. The vector $(m,n)$ represents $\v$ in the basis $\{\e_1, \e_2\}$. By Theorem \ref{thm:veech group}
there is an element $M \in V(Z^\circ)$ when written in this basis carries $(1,1)$ to $(m,n)$. In standard Euclidean coordinates, this carries the horizontal vector $-\e_0=\e_1+\e_2$ to $\w$. By Lemma \ref{lem:lifting} $M$ also represents
an element of the Veech group $V(S)$. Since the horizontal direction is completely drift-periodic on $S$, the direction $\theta$ must also be completely drift-periodic.

Now consider the direction of $\v_1$, represented by $(1,0)$ in our usual basis. There is a single cylinder on $Z^\circ$ in direction the $\v_1$, and its core curve is homomorphic to $\alpha$ shown in Figure \ref{fig:torus fundamental group}. Observe that this cylinder lifts as a cylinder to $S$ since the curve $\alpha$ lifts to a closed curve;
see the middle of the figure. By normality again,
the direction of $\v_1$ is completely periodic on $S$.

Now consider (2). In this case $(m,n) \not \in \Xi$. Theorem \ref{thm:veech group} tells us there is an element $M \in V(Z^\circ)$ carrying $(1,0)$ to $(m,n)$. Repeating the argument above, we see that the direction $\theta$ is completely periodic for $S$.

To see the final statement suppose that $p \in S$ has a periodic or drift periodic trajectory under $F_\theta$ for some $\theta \in S^1$. Let $\bar p \in Y^\circ$ be the image of $p$ under the covering map to $Y^\circ$ as in Proposition  \ref{prop:covering}. Let $\bar F_\theta$ denote straight-line flow in direction $\theta$ on $Y^\circ$. In either case, $\bar p$ has a periodic trajectory under $\bar F_\theta$. But, $Y$ is a torus
and closed geodesics are parallel to vectors in $\Lambda$
because of the particular geometry of this torus. (Edges of $Y$ are glued by translations by vectors in $\Lambda$.)
\end{proof}



\begin{proof}[Proof of Theorem \ref{thm:lattice}]
First suppose $\theta=\frac{\w}{|\w|}$ where
$\w \in \Lambda_\vis = {\mathcal P} \cup {\mathcal D}$. Note that by construction, ${\mathcal P}$ and ${\mathcal D}$ are invariant under the order twelve
dihedral group preserving $\Lambda$.
So, we may assume without loss of generality that $\theta\in[\frac{\pi}3,\frac{2\pi}3]$
and restrict to considering trajectories in $X_{\theta,+}$; see Proposition \ref{prop: standardization}. As $\w \in \Lambda_\vis$, we have $\w=m \v_1+n \v_2$ where $\gcd(m,n)=1$. As $\theta\in[\frac{\pi}3,\frac{2\pi}3]$,
we have $\|\w\|_{\scriptsize\hexagon} \equiv 0 \pmod{3}$ if and only if $m\equiv n \pmod{3}$.
In view of the orbit equivalence provided by Theorem \ref{thm:orbit equivalence}, it follows from Corollary \ref{cor:lattice directions on S} that
$\theta$ is a completely drift-periodic direction if $\theta \in {\mathcal D}$ and $\theta$ is a completely periodic direction if $\theta \in {\mathcal P}$.

To prove the stated converses, suppose $\theta$ is not-parallel to a vector in $\Lambda$. We may again assume by rotational symmetry that $\theta\in(\frac{\pi}3,\frac{2\pi}3)$
and consider trajectories in $X_{\theta,+}$.
The last statement of Corollary \ref{cor:lattice directions on S} guarantees that the straight-line flow in direction $\theta$ on $S$ lacks both periodic and drift-periodic trajectories. From the orbit equivalence, so does the flow $T_{\theta,+}$.
\end{proof}

\comref{Page 25: Combinatorial period is not defined.}\compat{Added the following paragraph.}

It would be nice to know finer information about the periodic orbits. For example, the {\em combinatorial period} of a periodic tiling billiard trajectory is the number of polygons crossed by the trajectory in a period.

\begin{question}
For a trajectory whose direction is parallel to a vector in $\Lambda$, what is the combinatorial period of the corresponding periodic or drift-periodic trajectory?
\end{question}

 In \cite{icerm}, $\S$5 explores the trihexagonal tiling, and answers this question for several directions in the lattice. For example, Proposition 5.10 in \cite{icerm} says, in our notation, that for a direction of the form $-(3n-2)\v_0-(6n+3)\v_1$, with $n\geq 1$, the trajectory is drift-periodic with combinatorial period $12n-6$. As $n$ increases, these approach the vertical direction. These directions correspond to the white vertices that are just to the right of the vertical line through the origin, on the first, third, fifth, etc. red hexagons in Figure \ref{fig:concentric-hexagons}. They have a similar result for trajectories approaching the direction $\pi/3$, our vector $\v_2$ (\cite{icerm}, Proposition 5.7).

The answer to this question is well understood for the square torus and square billiard table, and is explored for the double pentagon surface and the pentagonal table in \cite{DL}.

\subsection{Geometry of periodic and drift-periodic trajectories}

Since $S$ is a $\Z^2$-cover of $Z^\circ$, a cylinder $C \subset Z^\circ$ either lifts to a cylinder in $S$ or the universal cover embeds into $S$. In the later, we call this embedded image in $S$ a {\em strip} and denote it by $\tilde C$. For each such strip, there is pair of opposite elements $\pm \w \in 2 \Lambda$ which when acting on $S$ as an element in the Deck group preserves the strip $\tilde C$ and whose action on the strip generates the deck group of the covering $\tilde C \to C$. Note that
given $C \subset Z^\circ$, there are multiple choices of
a lift of $C$, but if one of these lifts is a strip then they all are and the pair $\pm \w$ only depends on $C$
so we denote it by $\pm \w(C)$. We call $\pm \w(C)$ the {\em deck group generators} of $C$.

If $C \subset Z^\circ$ is a cylinder a {\em holonomy vector} $\hol(C) \in \R^2$ of $C$ is a vector parallel to a core curve of $C$ with length equal to the length that curve.

\begin{proposition}
\label{prop:monodromy of strips}
Let $\theta=\frac{\v}{|\v|}$ be a drift-periodic direction on $S$ with $\v=m \e_1 + n \e_2 \in \Lambda_\vis$.
Then $Z^\circ$ has a decomposition into three cylinders of equal area in direction $\theta$. These three cylinders $\v$ as a holonomy vector and can be indexed so that $\pm \w(C_i)=\pm 2 \e_i$ for $i \in \{0,1,2\}$.
\end{proposition}
\begin{proof}
Fix $\theta$ and $\v$ as above. By Theorem \ref{thm:veech group}, there is an element $M \in V(Z^\circ)$ which carries $\v_0=(1,0)$ to $\v$. The holonomies of the three cylinders in the horizontal direction on $Z^\circ$ are each $\v_0$,
and it follows that the holonomy of vectors in direction $\theta$ are given by $M(\v_0)=\v$.

Let $\gamma_0=\alpha+\beta_0$, $\gamma_1=\beta_1$ and $\gamma_2=\beta_2$ denote the homology classes of core curves of the three horizontal cylinders; see Figure \ref{fig:torus fundamental group}.
 By definition of $h$ in \eqref{eq:h}, the deck group generators of the horizontal cylinders are given by $\pm h(\gamma_i)=\pm 2 \v_i$
for each $i$. Let $f:Z^\circ \to Z^\circ$ be an affine homeomorphism with derivative $M$. Then the core curves of cylinders in direction $\v$ have homology classes given by $f_\ast(\gamma_i)$. Then by Proposition \ref{prop:monodromy},
$$h \circ f_\ast(\gamma_i)=2 \big(\eta_0 \cap f_\ast(\gamma_i)\big)\v_0+2\big(\eta_1 \cap f_\ast(\gamma_i)\big)\v_1=
2 \big(f_\ast^{-1}(\eta_0) \cap \gamma_i\big)\v_0+2\big(f_\ast^{-1}(\eta_1) \cap \gamma_i\big)\v_1.
$$
Since $f^{-1}$ is an affine homeomorphism,
Proposition \ref{prop:six special classes}
guarantees that $f_\ast^{-1}(\eta_0)$ and $f_\ast^{-1}(\eta_1)$ each take values of the form $\pm \eta_0$,
$\pm \eta_1$ or $\pm(\eta_0-\eta_1)$. Using this and linearity it follows that $h \circ f_\ast(\gamma_i)$.
We have
$$\{\pm \w(C_i):~i=0,1,2\}=\{\pm h \circ f_\ast(\gamma_i):~i=0,1,2\}=\{\pm 2 \v_0, \pm 2 \v_1, \pm 2 \v_2\}.$$
\end{proof}

\begin{corollary}\label{cor:drift-periodic-shift}
Drift-periodic trajectories of the tiling billiard $T:X \to X$ are preserved by translation by $2\v_0$, $2\v_1$ or $2\v_2$. For any drift periodic direction $\theta \in {\mathcal D}$ and any sign $s$, there are trajectories of $T_{\theta,s}$ which are invariant under each of $2\v_0$, $2\v_1$ and $2\v_2$.
\end{corollary}

Notice that the drift-periodic trajectory on the right side of Figure \ref{fig:trajectory} is preserved by a horizontal translation of distance $2$; the preceding Corollary shows that every drift-periodic trajectory has this property, in one of the three directions parallel to edges of the tiling. This means that as we ``zoom out'' from a drift-periodic trajectory, it will converge to a line parallel
to one of the three directions $\v_0$, $\v_1$ or $\v_2$.

Recall that the torus $Z$ is a triple cover of the torus $Y$
formed from a single rhombus; see Figure \ref{fig:tori}.
So we have a chain of covers $S \to Z^\circ \to Y^\circ$,
where $Y^\circ$ represents $Y$ punctured at the identified vertices of the rhombus. The covers are all regular and
the covering $S \to Y^\circ$ has a deck group which is isomorphic to a semidirect product of $Z^2$ and $\Z/3\Z$: Concretely, this deck group is the group of orientation preserving isometries of the plane which permute the rhombi
in the decomposition of the hexagons while respecting
the notion of direction on the rhombi (see Figure \ref{fig:rhombi} and the discussion of direction below Corollary \ref{cor:translation surface}). Elements of this deck group  either have order three or are translations.

\begin{proposition}
\label{prop:S order 3}
Each cylinder of $S$ is invariant under an order three element of the deck group of the covering map $S \to Y^\circ$.
\end{proposition}
\begin{proof}
Let $C$ be a maximal cylinder of $S$. Then by Corollary \ref{cor:lattice directions on S},
core curves of $C$ are parallel to a vector
$\w=m \e_1+n \e_2 \in \Lambda_\vis$
with $m \not \equiv n \pmod{3}$. By Theorem \ref{thm:veech group} there is an element $M \in V(Z^\circ)=V(S)$ so that $M(\v_1)=\w$. Let $f:S \to S$
be an affine automorphism with derivative $M$. Then
$f^{-1}(C)$ is a maximal cylinder of $S$ in direction $\v_1$. Such cylinders in direction $\v_1$ consist of three rhombi, which are arranged symmetrically around a downward triangle in the tiling. (Triples of rhombi of this form are colored in the same way in Figure \ref{fig:pinwheel}.) Let $\delta:S \to S$ be the order three rotation of the tiling which permutes the three rhombi forming $f^{-1}(C)$. Then the cylinder $C$ is preserved by the order three element $f \circ \delta \circ f^{-1}$. \comref{The transformation you want seems inverse should be last, not first.}
\compat{The referee was right. It is now corrected.}
\end{proof}

\begin{corollary}\label{cor:periodic-order-three}
Each periodic trajectory of the tiling billiard $T:X \to X$ is invariant under a rotational symmetry of the tiling of period three but is not invariant under a rotational symmetry of period six.
\end{corollary}

Notice the order-three rotational symmetry of the periodic trajectory on the right side of Figure \ref{fig:trajectory}; this Corollary shows that all periodic trajectories have this property.

\begin{proof}
Period three invariance follows from Proposition \ref{prop:S order 3} since the deck group corresponds to symmetries of the tiling as noted above. Order six symmetry is impossible because the portions of a trajectory inside hexagons only travel in three directions; see Lemma \ref{lem:three-angles}.
\end{proof}

\section{The ergodicity criterion of Hubert and Weiss}
\label{sect:hubert weiss}

\subsection{The ergodicity criterion}
\label{sect:ergodicity criterion}

For this section, we will temporarily work in a more general context. 
Let $X$ be a compact translation surface $X$ and let $\Sigma \subset X$ be a finite collection of points containing the cone singularities of $X$. Let $S$ be a $\G$-cover of a compact translation surface $X^\circ  = X \smallsetminus \Sigma$. That is 
$S$ is a cover of $X^\circ$ with deck group $\G$ and $\G \backslash S= X^\circ$. We will only consider the case where $\G$ is abelian.

A cylinder in $X^\circ$ is said to be {\em maximal} if it is not contained in any other cylinder, i.e., it can not be nested inside another cylinder with the same circumference and holonomy. If $C \subset X^\circ$ is a cylinder, then the preimage of $C$
under the covering map $S \to X^\circ$ is either a disjoint union of cylinders or a disjoint union of strips.
Fix $C$ and choose a closed curve $\gamma:[0,1] \to C$ running once around the cylinder. Let $\tilde \gamma:[0,1] \to S$ be a lift. Then there is a $G_C \in G$ so that $G_C \cdot \tilde \gamma(0)= \tilde \gamma(1)$. 
Since $\G$ is abelian, the value of $G_C$ depends only on the direction $\gamma$ wraps around $C$. The cylinder $C$ lifts to strips if and only if 
$\langle G_C \rangle \subset \G$ is isomorphic to $\Z$. We suppress $\gamma$ from our notation for $G_C$ and just assume that our cylinders come with the choice of an isotopy class of oriented core curves.

We have the following definition:
\begin{definition}
Let $\theta$ be a direction for straight-line flow and let $G \in \G$.
We say the pair $(\theta, G)$ is {\em well-approximated by strips} on $S$ if there is an $\epsilon>0$ and infinitely many maximal cylinders $C$ on $X^\circ$ 
with areas bounded from below by some positive number
satisfy $G_C=G$ and
\begin{equation}
\label{eq: well-approximated}
|\hol~C| \cdot \big|\u_\theta \wedge \hol(C)\big| \leq (1-\epsilon) \Area(C).
\end{equation}
\end{definition}
An equivalent definition appears in \cite[\S 1]{hubert weiss} and \cite[Def. 5]{artigiani}
except that in those articles the right side of \eqref{eq: well-approximated} was scaled by a factor of $\frac{1}{2}$. 

We are also not assuming that $C$ lifts to strips because it is unnecessary for the proof and allowing torsion may be useful in some cases. (The work of
Hubert and Weiss
\cite{hubert weiss} considered $\G \cong \Z$ while Artigiani \cite{artigiani} considered $\Z^2$. The generalization to abelian groups
clearly follows.) In the case $G$ has finite order
``well-approximated by strips'' should be ``well-approximated by finite covers of cylinders'' but we will not concern ourselves with semantics.

We have the following which is a strengthening of \cite[Theorem 1]{hubert weiss} and \cite[Proposition 7]{artigiani}
because of the aforementioned lack of $\frac{1}{2}$ in \eqref{eq: well-approximated}). We thank Barak Weiss for pointing out that this factor of $\frac{1}{2}$ could be removed.

\begin{theorem}[Hubert-Weiss ergodicity criterion]
\label{thm:ergodicity criterion}
Suppose $S \to X$ is a $\G$-cover where $\G$ is abelian as above. Suppose $\langle G_1, \ldots, G_n \rangle$ is a finite index subgroup of $\G$ and that $\theta$ is an ergodic direction for straight-line flow on the finite cover of $X$ obtained by $\langle G_1, \ldots, G_n \rangle \backslash S$.
If for each $i$, $(\theta,G_i)$ is well-approximated by strips on $S$, then $\theta$ is an ergodic direction on $S$.
\end{theorem}

\subsection{Proof of the ergodicity criterion}
We will follow the approach of Hubert and Weiss. See in particular \S 2.4 and \S 3.1 of \cite{hubert weiss}.
Let $F^s:X^\circ \to X^\circ$ and $\tilde F^s:S \to S$ be the straight-line flows in direction $\theta$. We do not define the flows through $\Sigma$, but this affects only a set of zero measure.

The flow $\tilde F^s$ is measurably conjugate to a $\G$-valued skew product over the flow $F^s$.  To see this, select a basepoint $x_0 \in X^\circ$ and make a choice of a path $\beta_{x}$ in $X^\circ$ starting at $x_0$ and ending at $x$. Then for any $x \in X^\circ$ and $s \in S$ we may define the loop $\gamma_{x,s}:[0,1] \to X^\circ$ by first following $\beta_{x}$ then following the trajectory $F^{[0,s]}(x)$ forward (or $F^{[s,0]}(x)$ backward if $s<0$) and finally moving backward over $\beta_{F^s(x)}$ returning to $x_0$.  Select a lift $\tilde x_0 \in S$ of $x_0 \in X^\circ$, and let $\tilde \gamma_{x,s}:[0,1] \to S$ be the lift of $\gamma_{x,s}$ that begins at $\tilde x_0$. We define the cocycle
$$\alpha: X \times \R \to \G \quad \text{so that $\alpha(x,s) \in \G$ satisfies} \quad
\tilde \gamma_{x,s}(1)=\alpha(x,s)(\tilde x_0).$$
This choice makes $\tilde F^s$ is measurably conjugate to the skew product
$$\bar F^s: X \times \G \to X \times \G; \quad \bar F^s(x,G) = \big(F^s(x), G+\alpha(x,s)\big).$$

Let $\mu$ and be Lebesgue measure on $X$ and $\tilde \mu$ be the measure on $X \times \G$ which is the product of $\mu$ and the counting measure on $\G$. The conjugacy between $\tilde F^s$ and $\bar F^s$ carries $\tilde \mu$ to Lebesgue measure on $S$. 

Let $\G$ be a discrete group.
A group element $G \in \G$ is an {\em essential value} for the cocycle $\alpha$ if for any $A \subset X$ with $\mu(A)>0$, there is a set of $s$ in $\R$ of Lebesgue positive measure for which 
$$\mu \big( \{x \in A:~~\textrm{$F^s(x) \in A$~~and~~$\alpha(x,s)=G$}\}\big) >0.$$
The following is a consequence of work of Schmidt \cite[Cor. 5.4]{schmidt}:

\begin{theorem}
Assume $\G$ is a discrete abelian abelian group and $G_1, \ldots, G_k \in \G$ are essential values. Then $\bar F^s$ is ergodic if and only if 
the induced action of $\bar F^s$ on $X \times (\langle G_1, \ldots, G_k\rangle \backslash \G)$ is ergodic.
\end{theorem}

Note that we are stated the definition of essential value and the above theorem in the context of discrete groups, both this definition and result above extend with some modification to the indiscrete setting. See \cite{schmidt} for details.

Theorem \ref{thm:ergodicity criterion} then follows from the following:

\begin{lemma}
\label{lem:essential value}
If $(\theta, G)$ is well-approximated by strips, then $G$ is an essential value.
\end{lemma}

The remainder of this subsection is devoted to the proof of this lemma.

To simplify our arguments, rotate the surface so that $\theta$ is horizontal. Fix $G \in \G$
and assume $(\theta, G)$ is well approximated by strips. This guarantees the existence of an $\epsilon>0$ and a sequence of distinct maximal cylinders $C_n \subset X^\circ$
with area bounded from below so that $G_{C_n} = G$ and so that by defining the width of $C_n$ to be $w_n>0$ and $\hol(C_n)=(a_n,b_n)$ 
we have
\begin{equation}
\label{eq: well-approximated 2}
|b_n| \leq (1-\epsilon) w_n.
\end{equation}
(The equation above is equivalent in this context to \eqref{eq: well-approximated}.)
Given any bound, a translation surface has only finitely many maximal cylinders whose circumference is below this bound. Thus we have 
\begin{equation}
\label{eq:asymptotics}
\lim_{n \to \infty} w_n=0, \quad \lim_{n \to \infty} b_n=0, \quad \text{and} \quad
\lim_{n \to \infty} |a_n| = +\infty.
\end{equation}

For any $r<1$ and any $n$, let $C_n(r)$ be the set of points $x \in C_n$ for which the open metric ball $B(x; \frac{1}{2} r w_n)$ centered at $x$ and of radius $\frac{1}{2} r w_n$ is contained in the interior of $C_n$. Observe that $C_n(r)$ is the closed central cylinder of $C_n$ of width $(1-r) w_n$ and in particular
\begin{equation}
\label{eq:area comparison}
\Area~C_n(r)=(1-r) \Area~C_n.
\end{equation}

\begin{proposition}
For any $r$ with $0<r<1$, the set $\limsup_{n \to \infty} C_n(r)$ has full Lebesgue measure on $X$. 
\end{proposition}
The proof mirrors the proof of \cite[Lemma 14]{hubert weiss}. 

\begin{proof}
Let $m$ denote the Lebesgue measure on $X$. Set $L_r = \limsup_{n \to \infty} C_n(r)$. 
From the lower bound on areas of cylinders and \eqref{eq:area comparison}, we see that that $\mu(L_r) \geq \limsup \mu\big(C_n(r)\big)>0$ for all $r$.

The key observation is that for every $s \in \R$ and every $r'<r$ we have
$$F^s\big(C_n(r)\big) \subset C_n(r') \quad \text{for $n$ sufficiently large,}$$
where $F$ denotes the flow on $X$ in direction $\theta$. 
Indeed if $A \subset B$ are nested cylinders and $\v=\hol(B)$, then the image of $A$ under straight-line flow by vector $\w$ 
depends only on the projection of $\w$ onto the direction orthogonal to $\v$. In our setting we see $F^s\big(C_n(r)\big)$
is the same as the image of $C_n(r)$ under straight-line flow by the vector
$$\mathit{proj}_{(b_n,-a_n)} (s,0) \quad \text{which has norm} \quad \frac{|s b_n|}{\sqrt{a_n^2+b_n^2}} \leq \frac{(1-\epsilon)|s|w_n}{\sqrt{a_n^2+b_n^2}}$$
(as long as the images of $C_n(r)$ stay contained in $C_n$). Recalling that $C_n(r')$ is the central cylinder in $C_n$ of width $(1-r')w_n$ we see that
$$F^s\big(C_n(r)\big) \subset C_n(r') \quad \text{whenever} \quad 0<r'<r-\frac{2(1-\epsilon)|s|}{\sqrt{a_n^2+b_n^2}}.$$
From the limiting information of \eqref{eq:asymptotics}, we see that our key observation holds.

From the key observation it follows that for each $s$ and each $r'<r$ we have
$$F^s(L_r)=\limsup_n F^s\big(C_n(r)\big) \subset \limsup_n C_n(r') = L_{r'}.$$
This holds independent of $s$ so we see that that the $F^s$-orbit of $L_r$ is contained in $L_{r'}$.
By ergodicity of the flow in direction $\theta$ we have that $\mu(L_{r'})=\mu(X)$ since $\mu(L_{r'})>0$. 
This holds for all $r'$.
\end{proof}

\begin{proof}[Proof of Lemma \ref{lem:essential value}]
Continue using the notation as above. For this proof fix $r$ to be the value $1-\frac{\epsilon}{2}$ (where $\epsilon$ comes from the definition of well-approximation).
Let $L=\limsup_{n \to \infty} C_n(r)$ which has full measure from the prior proposition. To verify that $(\theta, G)$ is an essential value, fix a measurable
$A \subset X$ with $\mu(A)>0$. Let $x$ be a density point of $A$ which also lies in the full measure set $L$. That is, we insist
\begin{equation}
\label{eq:density}
\lim_{R \to 0} \frac{\mu\big(B(x;R) \cap A\big)}{\mu\big(B(x;R)\big)}=1.
\end{equation}

Since $a \in L$, there is an increasing sequence of integers $n_k$ so that $a \in C_{n_k}(r)$ for each $k$. We set
$B_k = B(x; \frac{1}{2} r w_{n_k})$ and by definition of $C_{n_k}(r)$ we see that $B_k \subset C_{n_k}$. Set $A_k = A \cap B_k$.
Observe that by \eqref{eq:asymptotics}, the radii of $B_k$ tends to zero as $k \to \infty$ so that $\lim_{k \to \infty} \mu(A_k)/\mu(B_k)=1$.

Let $\tilde B_k \subset S$ be a lift of $B_k$ and let $\tilde A_k \subset \tilde B_k$ be the preimage of $A_k$. We will argue that for values of $s$ nearby $a_{n_k}$, there is a large intersection between $\tilde F^{s}(\tilde B_k)$ and $G(\tilde B_k)$. See Figure \ref{fig:density}.
Then for sufficiently large $k$, density kicks in and implies that $\tilde F^{s}(\tilde A_k)$ and $G(\tilde A_k)$
intersect in a set of positive measure. Translating this in terms of cocycles, we see that this implies that $G$ is an essential value.

\begin{figure}
\includegraphics[width=5in]{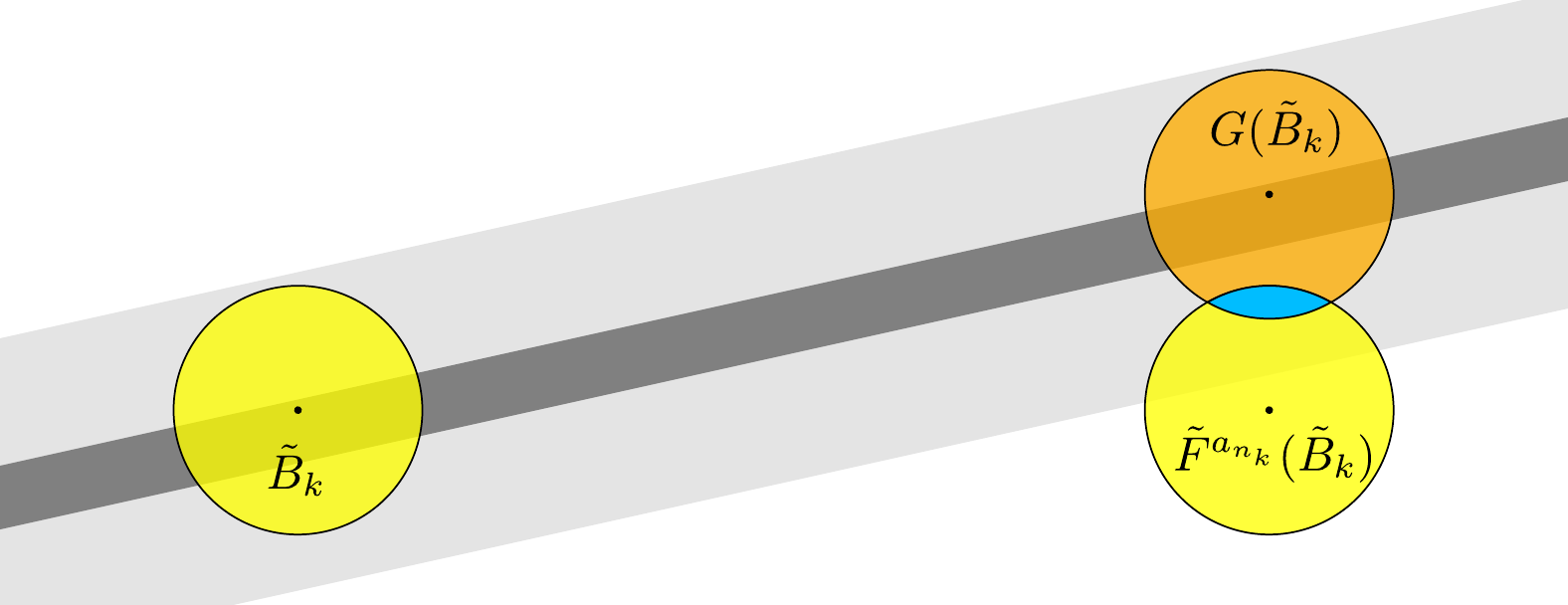}
\caption{In light grey we depict the strip $\tilde C_k$, and in dark grey the lift of $C_{n_k}(r)$. The intersection between $\tilde F^{a_{n_k}}(\tilde B_k)$ and $G(\tilde B_k)$ is shown.}
\label{fig:density}
\end{figure}

We will complete the proof by demonstrating that the intersection $\tilde F^{s}(\tilde B_k) \cap G(\tilde B_k)$ is large when $k$ is large. 
Let $\tilde C_k \subset S$ be the connected component of the preimage of $C_{n_k}$ containing $\tilde B_k$. 
Typically $\tilde C_k$ will be a strip and we assume this for simplicity. (If $\tilde C_k$ is not a strip we could work on the universal cover.)
Let $\dev: \tilde C_k \to \R^2$ be a developing map to the plane which is a translation in local coordinate charts. This is a homeomorphism to a strip in the plane. The action of $G$ translates along the strip $\tilde C_k$ and in fact:
$$
\dev \circ G(\tilde B_k)=(a_{n_k},b_{n_k})+\dev(\tilde B_k)
\quad \text{and} \quad
\dev \circ \tilde F^s(\tilde B_k \cap \tilde C_k) \supset \big((s,0)+\dev(B_k)\big)\cap \tilde C_k.$$
Considering the particular case of $s=a_{n_k}$, we see that the distance between the two circles $(a_{n_k},b_{n_k})+\dev(\tilde B_k)$
and $(s,0)+\dev(B_k)$ differ by a vertical translation
by $|b_{n_k}|$. Thus from \eqref{eq: well-approximated 2} and consideration of the radius of $\tilde B_k$ we see that the region
\begin{equation}
\label{eq:lune}
\dev \circ G(\tilde B_k) \cap \dev \circ \tilde F^{a_{n_k}}(\tilde B_k \cap \tilde C_k)
\end{equation}
has height
$$r w_{n_k}-|b_{n_k}|=(1-\frac{\epsilon}{2})w_{n_k} - |b_{n_k}| \geq \frac{\epsilon}{2} w_{n_k}$$
(where we use our definition of $r$ and \eqref{eq: well-approximated 2}), which is a positive proportion of the radius $r w_{n_k}$ of $B_k$.
In particular, there is an $\eta>0$ so that
$$\frac{\tilde \mu\big(G(\tilde B_k) \cap \tilde F^{a_{n_k}}(\tilde B_k)\big)}{\tilde \mu(\tilde B_k)}>\eta
\quad \text{for all $k$}.$$
Note that this ratio of areas varies continuously as $s$ varies near $a_{n_k}$, so this ratio is nearly as large when $s$ is near $a_{n_k}$.
By the remarks of the previous paragraph, this completes the proof.
\end{proof}

\subsection{A geometric interpretation}

We will now give a geometric interpretation of the concept of well-approximation by strips in the case where the base surface $X^\circ$ has Veech's lattice property. Essentially this amounts to working out a statement of
\cite[Proposition 5]{hubert weiss} which was left to the reader with attention to the explicit bounds.

We begin with some hyperbolic geometry. We continue to follow the conventions established in \S \ref{sect:hyperbolic geometry}.
Recall $\H^2=O(2)\bs \PGL(2,\R)$. For $M \in \PGL(2,\R)$ we use $[M] \in \H^2$ to denote the associated coset.
For each vector $\v \in \R^2$,
we may define the Busemann function
$$B_\v: \H^2 \to \R; \quad [M] \mapsto \log |M\v|.$$
\comref{Page 28: Busemann functions are defined to have image into $R>0$, and then it is said that $Mn \to v$ iff $Bv(Mn)\to -\infty$.}\compat{Previously the equation above had image $\R_{>0}$. The $>0$ was erroneous.}
If $\v$ has unit length this coincides with the usual notion of the Busemann function for the geodesic $t \mapsto [g_t r_{\theta}^{-1}]$ where $r_{\theta}\in \SO(2,\R)$ is a rotation so that $r_\theta^{-1}(\v)$ is horizontal and $g_t$ is as in \eqref{eq:gt}. The Busemann functions yield a natural compactification of $\H^2$ where
$[M_n] \to [\v]$ if and only if $B_\v([M_n]) \to - \infty$. A {\em horodisk neighborhood} of $[\v] \in \RP^2$ is a set of the form
\begin{equation}
\label{eq:horodisk}
H(\v, \epsilon)=\{[M] \in \H^2~:~\exp \circ B_\v([M])<\epsilon\}
\quad \text{for some $\epsilon>0$}.
\end{equation}

Let $\Gamma \subset \SL(2,\R)$ be a non-uniform lattice and $P \in \Gamma$ be a parabolic preserving the eigenvector $\v \in \R^2$. Then the quotient $\H^2/\Gamma$ has a cusp associated to $P$ and $\v$. We define
${\mathcal C}(\v,\epsilon)$ to be the image of the horodisk $H(\v,\epsilon)$ in $\H^2/\Gamma$.
It is useful to observe that for any $h \in \SL(2,\R)$ and any $\epsilon>0$ we have
$M \in H\big(h(\v),\epsilon\big)$ if and only if $M \in H(\v,\epsilon)\cdot h^{-1}$
and thus
\begin{equation}
\label{eq:H orbit}
H(\v,\epsilon) \cdot \Gamma=\bigcup_{\gamma \in \Gamma} H\big(\gamma(\v),\epsilon\big)
\end{equation}
or equivalently ${\mathcal C}(\v,\epsilon)={\mathcal C}\big(\gamma(\v),\epsilon\big)$ for all $\gamma \in \Gamma$.

\begin{lemma}
\label{lem:equivalence}
Let $\Gamma \subset \SL(2,\R)$ be a discrete group
and let $\v \in \R^2$ be an eigenvector of a parabolic $P \in \Gamma$.
Let $\u_\theta \in \R^2$ be the unit vector in direction $\theta$. Then the following statements
are equivalent for any $d>0$:
\begin{enumerate}
\item[(a)] There is a sequence $\gamma_n \in \Gamma$ so that the vectors $\v_n=\gamma_n(\v)$ are pairwise distinct and satisfy
$$\liminf_{n \to \infty} |\v_n|\cdot |\u_\theta \wedge \v_n| < d,$$
where $\u_\theta \wedge \v_n$ is the signed area of the parallelogram with the two vectors as edges.
\item[(b)] The geodesic ray
$[0,+\infty) \to \H^2/\Gamma$ defined by $t \mapsto [g_t r_\theta^{-1}\Gamma]$ has an accumulation point in the open neighborhood of the cusp ${\mathcal C}(\v,\sqrt{2d})$.
\end{enumerate}
\end{lemma}

Statement (a) is related to the notion of well-approximation by strips, see \eqref{eq: well-approximated}.

Lemma \ref{lem:equivalence} gives a geometric criterion for deducing well-approximated by strips in the case when the Veech group is non-elementary.
\begin{proposition}
\label{prop:criterion}
Let $S$ be a $\G$-cover of $X^\circ$ where $\G$ is abelian as in \S \ref{sect:ergodicity criterion}.
Let $\Gamma \subset \SL(2,\R)$ denote the subgroup of the Veech group of $S$ consisting of derivatives of affine automorphisms which commute with all elements of the deck group $\G$. Let $C$ be a cylinder on $X$ which lifts to a strip in $S$
so that $\hol~C$ is preserved by a parabolic $P \in \Gamma$. Then, if $\theta$ is a direction so that the geodesic ray $[g_t r^{-1}_\theta \Gamma]$ has an accumulation point in the open neighborhood of the cusp ${\mathcal C}\big(\hol~C,\sqrt{2 \Area(C)}\big)$ in the surface $\H^2/\Gamma$ then $(\theta, G_C)$ is well-approximated by strips on $S$.
\end{proposition}
\begin{proof}[Proof assuming Lemma \ref{lem:equivalence}]
The commutativity assumption tells us that for any $\gamma \in \Gamma$, there is an affine automorphism $\phi$ with derivative $\gamma$ so that the cylinder $\phi(C)$ satisfies $G_{\phi(C)}=G_C$. Such cylinders are therefore available to use for verifying that $(\theta, G_C)$ is well-approximated by strips. By hypothesis we have statement (b) of Lemma \ref{lem:equivalence} for $\v=\hol~C$ and
$d=\Area(C)$. Since (b) implies (a), there is a sequence $\gamma_n \in \Gamma$ so that
$\v_n=\gamma_n(\v)$ are pairwise distinct and so that
$$\liminf_{n \to \infty} |\v_n|\cdot |\u_\theta \wedge \v_n|<d.$$
Let $K$ denote the $\liminf$ above.
For each $n$, choose an affine automorphism $\phi_n$ with derivative $\gamma_n$ so that $\phi_n$
commutes with all elements of $\G$. Set $C_n=\phi_n(C)$ which are cylinders with the same area and $G_{C_n}=G_C$ from our first remarks. Observe that $|\hol~C_n| \cdot \big|\u_\theta \wedge \hol(C_n)\big|$
coincides with $|\v_n|\cdot |\u_\theta \wedge \v_n|$ and so the limit inferior of this quantity is $K$.
The definition of well-approximated is therefore satisfied when $0<\epsilon<1-\frac{K}{\Area(C)}$. \end{proof}


The remainder of the section is devoted to the proof of Lemma \ref{lem:equivalence}. In order to simplify the argument, we assume without loss of generality that the direction $\theta$ is horizontal and $\u_\theta=(1,0)$. Denote $\v_n$ by $(x_n,y_n)$. Then we have
\begin{equation}
\label{eq:trivial identity}
|\v_n|\cdot |\u_{\theta} \wedge \v_n|=|y_n|\sqrt{x_n^2 + y_n^2}.
\end{equation}
It will be useful for us to observe:

\begin{proposition}
\label{prop:asymptotic}
For any sequence $\v_n=(x_n,y_n)$ so that $[\v_n] \to [(1,0)]$ in $\RP^2=\R^2/(\R \smallsetminus \{0\})$,
$$\liminf_{n \to \infty} |x_n y_n| = \liminf_{n \to \infty} |y_n|\sqrt{x_n^2+y_n^2}.
$$
\end{proposition}

The proof is elementary and left to the reader.
\compat{A proof (that could use improvement) is commented out here.}

\begin{proposition}
\label{prop:horoball}
Let $\w=(x,y)$ be a a vector in $\R^2$ with $|x| > |y| > 0$. The geodesic ray $\{[g_t]:t>0\}$
is tangent to the horoball $H(\w,\sqrt{2|x||y|})$.
\end{proposition}
\begin{proof}
By a direct calculation we observe that $[g_t]$ is in the boundary of the horoball $H(\w,e^{-t}x+e^t y)$.
The minimum attained by $e^{-t}x+e^t y$ taken over all $t \in \R$ is $\sqrt{2|x||y|}$
and this is attained for the $t$ satisfying $|e^{-t} x|=|e^t y|$ which occurs for $t>0$ since $|x| > |y| > 0$.
\end{proof}

\begin{proof}[Proof of Lemma \ref{lem:equivalence}]
Fix the parabolic $P \in \Gamma$ with eigenvector $\v$ as in the statement of the theorem. Assume without loss of generality that $\theta$ is horizontal as mentioned above. Fix a $d>0$.

Suppose we have sequences $\gamma_n$ so that $\v_n=\gamma_n(\v)=(x_n,y_n)$ satisfies
statement (a) of the theorem. Recall that because $\Gamma$ is a discrete group, the orbit of the eigenvector $\v$ of the parabolic $P \in \Gamma$ is discrete. This forces us to have $|\v_n| \to \infty$,
and it follows from \eqref{eq:trivial identity} and our hypothesis on $\{v_n\}$ that
 $[\v_n] \to [(1,0)]$ and $|x_n| \to \infty$ where $\v_n=(x_n,y_n)$. By possibly removing finitely many terms, we may assume $|x_n|>|y_n|$ for all $n$. Let $K = \liminf_{n \to \infty} |\v_n|\cdot |\v_n \wedge \u_\theta|$ and recall that $K<d$ by hypothesis.
By Proposition \ref{prop:asymptotic} and \eqref{eq:trivial identity}, we know
$\liminf_{n \to \infty} |x_n y_n| = K$.
As a consequence of Proposition \ref{prop:horoball}, we know that the ray $\{[g_t]:~t>0\}$ is tangent to the horoball $H(\v_n,\sqrt{2|x_n||y_n|})$ for each $n$. On the surface this horoball descends to the cusp neighborhood ${\mathcal C}(\v, \sqrt{2|x_n||y_n|}))$.
As $d>K$, there are infinitely many $n$ so that the ray visits the closure of the horoball
$H(\v_n,\sqrt{K+d})$ so we get an accumulation point in the closure
of ${\mathcal C}(\v, \sqrt{K+d})$ which is contained in ${\mathcal C}(\v, \sqrt{2d})$.

Now suppose that the ray $\{[g_t]:~t>0\}$ has an accumulation point in the open cusp neighborhood ${\mathcal C}(\v, \sqrt{2d})$. This accumulation point lies in the boundary of ${\mathcal C}(\v, \sqrt{2K})$ for some $K<d$. That is, there is a sequence of times $t_n \to +\infty $ so that
$g_{t_n} \cdot \Gamma$ converges to a point in the closure of ${\mathcal C}(\v, \sqrt{2K})$.
Lifting to $\SL(2,\R)$, we get a sequence $\gamma_n \in \Gamma$ so that $g_{t_n} \gamma_n$ converges to a point in the $H(\v,\sqrt{2K})$. In particular, the ray
$\{g_{t} \gamma_n~:~t>0\}$ passes through
the open horoball $H(\v,\sqrt{K+d})$. From the discussion above \eqref{eq:H orbit},
this is equivalent to the statement that the ray $\{g_t:~t>0\}$ passes through $H\big(\gamma_n(\v),\sqrt{K+d}\big)$. Set $\v_n=\gamma_n(\v)=(x_n,y_n)$. This ray is tangent to the horoball $H\big(\v_n, \sqrt{2 |x_n| |y_n|})$ by Proposition \ref{prop:horoball}.
In particular then
$\sqrt{2|x_n| |y_n|}<\sqrt{K+d}$ for each $n$. It follows that $\liminf |x_n||y_n| \leq \frac{K+d}{2}$.
Then by Proposition \ref{prop:asymptotic},
$$\liminf |y_n|\big(|x_n|+|y_n|\big) \leq \frac{K+d}{2}<d.$$
\end{proof}

\section{Ergodicity of aperiodic directions}
\label{sect:ergodicity}

\comref{The final part of Section 6 heavily uses notation, terminology and results from the subsequent Appendix, this basically forces the reader to skip it and then come back to it. I think that you should either invert the two sections or anticipate the needed material to Section 6.}

Our Theorem \ref{thm:main ergodic} on the ergodicity of the tiling billiard flows is a consequence of an ergodicity result for the translation surface $S$. We show the following using methods introduced in the previous section.

\begin{theorem}
\label{thm:ergodicity on S}
Let $0 \leq \theta \leq \pi$ and let $\vec{u}_\theta \in T_1(\Delta)$ be the unit tangent vector based at $i$ which is tangent to the geodesic ray in the upper half-plane terminating at $|\!\cot \theta|$ as described above \eqref{eq:E0}. Consider the billiard trajectory $g_t(\vec{u}_\theta)$ in $\Delta$ as in the introduction. If
$$\limsup_{t \to +\infty} \Im\big(g_t(\vec{u}_\theta)\big) > \frac{1}{\sqrt{3}}
\quad \text{and} \quad
\liminf_{t \to +\infty} \Im\big(g_t(\vec{u}_\theta)\big) \neq +\infty,$$
then the straight-line flow $F_\theta:S \to S$ is ergodic.
\end{theorem}

The criterion we use to prove this is given in Theorem \ref{thm:ergodicity criterion}.
Our surface $S$ is a $\Z^2$ cover of $Z^\circ$.
We define $\Gamma$ to be the subset of the Veech group $V(S)$ consisting of elements $M \in \Gamma$ so that there is an affine automorphism $f:S \to S$ which commutes with all elements in $2 \Lambda$, the deck group of the cover $S \to Z^\circ$. To this end we show:

\begin{proposition}
\label{prop:commuting with the deck group}
An affine automorphism $f:Z^\circ \to Z^\circ$ has a lift to $S$ which commutes with all elements of the deck group of the cover $S \to Z^\circ$ if
$f$ preserves the set $\Sigma \subset Z$ of punctures
pointwise.
\end{proposition}
\begin{proof}
It follows from Proposition \ref{prop:six special classes} that the classes $\eta_0, \eta_1 \in H_1(S,\Sigma;\Z)$ are preserved by $f$. (The six classes are distinguished by their images in $H_0(\Sigma,\Z)$ under the boundary map.)

Let $p_0 \in Z^\circ$ be the basepoint for the curves $\alpha$, $\beta_0$, $\beta_1$ and $\beta_2$ on the left side of Figure \ref{fig:torus fundamental group}.
When the curves $\beta_0$ and $\beta_1$ are lifted to curves $\tilde \beta_0$ and $\tilde \beta_1$ on $S$, the endpoints differ by the deck group elements $-2 \v_0$ and $-2 \v_1$, respectively. Denote these two deck group elements by $D_0$ and $D_1$ respectively, and note that they generate the deck group $2 \Lambda$.

Let $\tilde f:S \to S$ be a lift of $f$ which exists
by Lemma \ref{lem:lifting}.
Since $D_i$ is the element of the deck group carrying the starting point of $\tilde \beta_i$ to the end point, it follows that $\tilde f \circ D_i \circ \tilde f^{-1}$ is the element of the deck group carrying the starting point of $\tilde f \circ \tilde \beta_i$ to its end point. The curve $\tilde f \circ \tilde \beta_i$ is a lift of the curve $f \circ \beta_i$ and therefore the deck group element $\tilde f \circ D_i \circ \tilde f^{-1}$ coincides with the monodromy $h(f \circ \beta_i)$; see \eqref{eq:h}. We can compute $h(f \circ \beta_i)$ using Proposition \ref{prop:monodromy}:
$$h(f \circ \beta_i)=2\big(\eta_0 \cap f_\ast(\beta_i) \big)\v_0+2\big(\eta_1 \cap f_\ast(\beta_i)\big)\v_1
=
2\big(f_\ast^{-1}(\eta_0) \cap \beta_i \big)\v_0+2\big(f_\ast^{-1}(\eta_1) \cap \beta_i\big)\v_1.$$
Since $f_\ast^{-1}$ stabilizes $\eta_0$ and $\eta_1$,
we see $h(f \circ \beta_i)=h(\beta_i)$ or equivalently
$\tilde f \circ D_i \circ \tilde f^{-1}=D_i$.
\end{proof}

\begin{corollary}
\label{cor:index six}
The subgroup $\Gamma \subset V(S)$ is at most index six.
\end{corollary}
\begin{proof}
The affine automorphisms satisfying Proposition \ref{prop:commuting with the deck group} lie in the kernel of a group homomorphism to the permutation group of $\Sigma$.
\end{proof}

\begin{proposition}
\label{prop:ergodicity on Z}
With $\theta$ satisfying the hypotheses of Theorem \ref{thm:ergodicity on S}, the straight-line flow on $Z^\circ$ in direction $\theta$ is uniquely ergodic.
\end{proposition}
\begin{proof}
By hypothesis the billiard trajectory $g_t(\vec{u}_\theta)$ has an accumulation point
in $\SL(2,\R)/\big(V(Z^\circ) \cap \SL(2,\R)\big)$. (This quotient is naturally the unit tangent bundle of the double of $\Delta$ across its boundary.)
By Masur's criterion \cite[Theorem 1.1]{masurs criterion}, straight-line flow on $Z^\circ$ in direction $\theta$ is uniquely ergodic.
\end{proof}

Recall that a cylinder $C$ which lifts to a strip has an element of the deck group $G_C$ associated to
it, see \S \ref{sect:ergodicity criterion}.

\begin{proposition}
\label{prop:horizontal}
In the horizontal direction, $Z^\circ$ has a horizontal cylinder decomposition consisting of three cylinders $C_0$, $C_1$, $C_2$ which lifts to strips in $S$. The cylinders can be indexed so that $\hol(C_i)=(1,0)$, $\Area(C_i)=\frac{\sqrt{3}}{2}$ and $G_{C_i}=- 2 \v_i$.
\end{proposition}
This follows from the geometric description of $Z^\circ$ and Proposition \ref{prop:monodromy}.

The cylinders lifting to strips as above determine
cusp neighborhoods as in Proposition \ref{prop:criterion}. The following gives the geometry of a corresponding horodisk defined in \eqref{eq:horodisk}.

\begin{proposition}
\label{prop:cusp neighborhood calculation}
The horodisk $H\big((1,0),\sqrt[4]{3}\big)$ is given by $\{z \in \C:~\Im~z>\frac{1}{\sqrt{3}}\}.$
\end{proposition}
\begin{proof}
Let $H$ denote $H\big((1,0),\sqrt[4]{3}\big)$.
By definition, for $M \in \PSL(2,\R)$, we have $[M] \in H$ if
$|M(1,0)|<\sqrt[4]{3}$.
Consider $M=g_t$ as in \eqref{eq:gt}. By \eqref{eq:upper half-plane}, $g_t$ is the unit tangent vector based at $e^{2t} i$. Also,
$|g_t(1,0)|=e^{-t}$ so that $e^{2t}i \in H$ if and only if $e^{-t}<\sqrt[4]{3}$. Thus
$\frac{1}{\sqrt{3}} i \in \partial H$. Since $H$
is a horodisk based at $\infty$, the proposition follows.
\end{proof}

\begin{proof}[Proof of Theorem \ref{thm:ergodicity on S}]
This is a consequence of Theorem \ref{thm:ergodicity criterion}. Fix $\theta$ as in the statement of the theorem.
The straight-line flow in direction $\theta$ is
ergodic on $Z^\circ$ by Proposition \ref{prop:ergodicity on Z}. Consider the deck group elements of $S$ associated to $2 \v_0$ and $2 \v_1$. We need to show that $(\theta,-2 \v_0)$ and $(\theta, -2 \v_1)$ are well-approximated by strips.
We use Proposition \ref{prop:criterion} for this.
The subgroup $\Gamma \subset V(Z^\circ)$ used in the Proposition is finite index by Corollary \ref{cor:index six}. In particular the parabolic $P \in V(Z^\circ)$ has a power which lies in $\Gamma$.

By Proposition \ref{prop:horizontal} the horizontal direction has three cylinders on $Z^\circ$ which satisfies $G_{C_i}=- 2 \v_i$ for $i=0,1,2$.
Proposition \ref{prop:criterion} then guarantees
that both $(\theta,-2 \v_0)$ and $(\theta, -2 \v_1)$ are well-approximated by strips if
there is an accumulation point of $[g_t r_\theta^{-1} \Gamma]$ in the cusp neighborhood ${\mathcal C}\big((1,0),\sqrt[4]{3}\big)$
of $\H^2/\Gamma$ which is defined as the projection of
$H\big((1,0),\sqrt[4]{3}\big)$
to $\H^2/\Gamma$.

The Veech group $V(Z^\circ)$ acts transitively on strip direction of $S$ by Theorem \ref{thm:veech group}. Also, all strip directions have oriented cylinders with
$G_C=- 2 \v_i$ for $i \in \{0,1,2\}$ by
Proposition \ref{prop:monodromy of strips}.
In particular, for all $M \in V(Z^\circ)$, if
$[g_t r_\theta^{-1} \Gamma]$ accumulates in a point in the cusp neighborhood ${\mathcal C}\big(M(1,0),\sqrt[4]{3}\big)$,
then we also get that both $(\theta,-2 \v_i)$ are well-approximated by strips for each $i \in \{0,1,2\}$.

Recall from Corollary \ref{cor:Delta} that
$\Delta=\H^2/V(Z^\circ)$. Since $\Gamma \subset V(Z^\circ)$, there is a natural covering map $p:\H^2/\Gamma \to \Delta$. The cusps neighborhoods found above consist of preimages under $p$ of the neighborhood consisting of points $z \in \Delta$ with
$\Im z > \frac{1}{\sqrt{3}}$. This follows from naturality of the definition of neighborhood (see \eqref{eq:H orbit}) and Proposition \ref{prop:cusp neighborhood calculation} in the case of the horizontal direction. Since the covering is finite-to-one, there is an accumulation point in one of these cusp neighborhoods if and only if there is an accumulation point of $[g_t r_\theta^{-1} V(Z^\circ)]$
in the cusp neighborhood of $\Delta$ described above.
This is guaranteed by hypothesis in the statement of the Theorem. Thus we get that $(\theta,-2 \v_0)$ and $(\theta, -2 \v_1)$ are well-approximated by strips
and therefore the flow in direction $\theta$ is ergodic on $S$ by Theorem \ref{thm:ergodicity criterion}.
\end{proof}

\begin{proof}[Proof of Theorem \ref{thm:main ergodic}]
By Proposition \ref{prop: standardization} and invariance of ${\mathcal E}$ under symmetries of the tiling, it suffices to consider the flow $T_{\theta,+}:X_{\theta,+} \to X_{\theta,+}$ for some $\theta \in {\mathcal E} \cap [\frac{\pi}{3}, \frac{2\pi}{3}].$ By Theorem \ref{thm:ergodicity on S}, the flow $F_\theta:S \to S$ is ergodic.
Consider the orbit-equivalence $\x:S \to X_{\theta,+}$
given by Theorem \ref{thm:orbit equivalence}.
Because the orbit equivalence is bilipschitz in the flow direction and measure preserving in the transverse direction, an $F_\theta$-invariant measurable set $A \subset S$ has measure zero
if and only if $\x(A) \subset X_{\theta,+}$ has measure zero. Thus $T_{\theta,+}$ is also ergodic.
\end{proof}

%
%
%

\compat{I updated the references that were published. You might update the papers you are co-author on if they are now published.}

\noindent Diana Davis, \\
Swarthmore College, Department of Mathematics and Statistics, 500 College Avenue, Swarthmore PA 19081 \\
\mbox{\url{ddavis3@swarthmore.edu }} \\

\noindent W. Patrick Hooper, \\
Department of Mathematics, The City College of New York, 160 Convent Ave, New York, NY 10031 \\
Department of Mathematics, CUNY Graduate Center, 365 5th Ave, New York, NY 10016 \\
\url{whooper@ccny.cuny.edu}\\

\end{document}